\documentclass[12pt,english]{smfart}

\usepackage{smfenum}
\usepackage{smfthm}
\usepackage{amssymb}
\usepackage{epic}

\newcommand{\val}{\operatorname{val}}
\newcommand{\Qp}{\mathbf{Q}_p}
\newcommand{\Qpn}{\mathbf{Q}_{p^n}}
\newcommand{\Zp}{\mathbf{Z}_p}
\newcommand{\Fp}{\mathbf{F}_p}
\newcommand{\Fpn}{\mathbf{F}_{p^n}}
\newcommand{\Cp}{\mathbf{C}_p}
\newcommand{\ZZ}{\mathbf{Z}}

\newcommand{\OO}{\mathcal{O}}

\newcommand{\Fpbar}{\overline{\mathbf{F}}_p}
\newcommand{\Qpbar}{\overline{\mathbf{Q}}_p}
\newcommand{\eps}{\varepsilon}
\renewcommand{\phi}{\varphi}
\renewcommand{\projlim}{\varprojlim}
\renewcommand{\geq}{\geqslant}
\renewcommand{\leq}{\leqslant} 
\newcommand{\G}{\operatorname{GL}_2(\Qp)}
\newcommand{\B}{\mathrm{B}}
\newcommand{\K}{\mathrm{K}}
\newcommand{\KK}{\operatorname{GL}_2(\Zp)}
\newcommand{\I}{\mathrm{I}_1}
\newcommand{\Z}{\mathrm{Z}}
\newcommand{\gal}{\mathcal{G}_{\Qp}}
\newcommand{\hal}{\mathcal{H}_{\Qp}}
\newcommand{\galpn}{\mathcal{G}_{\Qpn}}
\newcommand{\inert}{\mathcal{I}_{\Qp}}
\newcommand{\et}{\widetilde{\mathbf{E}}}
\newcommand{\ve}{\operatorname{val}_{\mathbf{E}}}
\newcommand{\ddiese}{\mathrm{D}^{\sharp}}
\newcommand{\dfont}{\mathrm{D}}
\newcommand{\ind}{\operatorname{ind}}
\newcommand{\Sym}{\operatorname{Sym}}
\newcommand{\Id}{\operatorname{Id}}
\newcommand{\res}{\operatorname{res}}
\newcommand{\Mat}{\operatorname{Mat}}
\newcommand{\Gal}{\operatorname{Gal}}
\newcommand{\GL}{\operatorname{GL}}
\newcommand{\smat}[1]{\left( \begin{smallmatrix} #1 \end{smallmatrix} \right)}
\newcommand{\pmat}[1]{\begin{pmatrix} #1 \end{pmatrix}}
\newcommand{\dpar}[1]{(\!( #1 )\!)}
\newcommand{\dcroc}[1]{[\![ #1 ]\!]}

\newcommand{\pscal}[1]{\langle #1 \rangle}

\author{Laurent Berger}
\address{Universit\'e de Lyon \\
UMPA ENS Lyon \\
46 all\'ee d'Italie \\
69007 Lyon \\
France}
\email{laurent.berger@umpa.ens-lyon.fr}
\urladdr{www.umpa.ens-lyon.fr/\~{}lberger/}

\date{November 2008, revised May 2009}

\title{On some modular representations of the 
Borel subgroup of $\operatorname{GL}_2(\mathbf{Q}_p)$}

\subjclass{05C05, 11F33, 11F70, 11F80, 11F85, 20C20, 20G25, 22E35, 22E50}

\keywords{$p$-adic Langlands correspondence, supersingular representations, $(\varphi,\Gamma)$-modules, Galois representations}

\NumberTheoremsIn{subsection}

\begin{document}

\begin{abstract}
Colmez has given a recipe to associate a smooth modular representation $\Omega(W)$ of the Borel subgroup of $\operatorname{GL}_2(\mathbf{Q}_p)$ to a $\overline{\mathbf{F}}_p$-representation $W$ of $\operatorname{Gal}(\overline{\mathbf{Q}}_p / \mathbf{Q}_p)$ by using Fontaine's theory of $(\varphi,\Gamma)$-modules. We compute $\Omega(W)$ explicitly and we prove that if $W$ is irreducible and $\dim(W)=2$, then $\Omega(W)$ is the restriction to the Borel subgroup of $\operatorname{GL}_2(\mathbf{Q}_p)$ of the supersingular representation associated to $W$ by Breuil's correspondence.
\end{abstract}

\begin{altabstract}
Colmez a donn\'e une recette permettant d'associer une repr\'esentation modulaire $\Omega(W)$ du sous-groupe de Borel de $\operatorname{GL}_2(\mathbf{Q}_p)$ \`a une $\overline{\mathbf{F}}_p$-repr\'esentation $W$ de $\operatorname{Gal}(\overline{\mathbf{Q}}_p / \mathbf{Q}_p)$ en utilisant la th\'eorie des $(\varphi,\Gamma)$-modules de Fontaine. Nous d\'eterminons $\Omega(W)$ explicitement et nous montrons que si $W$ est irr\'eductible et $\dim(W)=2$, alors $\Omega(W)$ est la restriction au sous-groupe de Borel de $\operatorname{GL}_2(\mathbf{Q}_p)$ de la repr\'esentation supersinguli\`ere associ\'ee \`a $W$ par la correspondance de Breuil.
\end{altabstract}

\maketitle

\setcounter{tocdepth}{2}
\tableofcontents

\setlength{\baselineskip}{18pt}

\section*{Introduction}
This article is a contribution to the $p$-adic Langlands correspondence, and more specifically the ``mod $p$'' correspondence first introduced by Breuil in \cite{BR1} which is a bijection between the supersingular representations of $\G$ and the irreducible $2$-dimensional $\Fpbar$-linear representations of $\gal = \operatorname{Gal}(\overline{\mathbf{Q}}_p / \mathbf{Q}_p)$. In \cite{CPG} and \cite{CGL}, Colmez has given a construction of representations of $\G$ associated to certain $p$-adic Galois representations and by specializing and extending his functor to the case of $\Fpbar$-representations, we get a recipe for constructing a smooth representation $\Omega(W)$ of the Borel subgroup $\B=\B_2(\Qp)$ of $\G$ starting from the data of an $\Fpbar$-representation $W$ of $\gal$. In \cite{LB5}, I proved that Colmez' construction was compatible with Breuil's mod $p$ correspondence and as a consequence that Colmez' $\projlim_\psi \ddiese(\cdot)$ functor in characteristic $p$ does give Breuil's correspondence (up to semisimplification if $W$ is reducible). The proof of \cite{LB5} is direct when $W$ is reducible (in which case $\Omega(W)$ is a parabolic induction) but quite indirect when $W$ is absolutely irreducible (in which case $\Omega(W)$ is supersingular) and one purpose of this article is to give a direct proof in the latter case. A byproduct of the computations of \cite{LB5} is the fact that the restriction to the Borel subgroup of a supersingular representation is still irreducible. This intriguing fact has since been reproved and generalized by Pa\v{s}k\={u}nas in \cite{P06} (see also \cite{EM}; another generalization has been worked out by Vign\'eras in \cite{V06}).

In this article, we start by defining some smooth representations of $\B$ and we prove that they are irreducible. After that, we define the representations $\Omega(W)$ using Colmez' functor applied to $W$  and finally, we prove that if $\dim(W) \geq 2$ and $W$ is irreducible, then the $\Omega(W)$ thus constructed coincide with the representations studied in the first chapter and that if $\dim(W)=2$, then they are the restriction to $\B$ of the supersingular representations studied by Barthel and Livn\'e in \cite{BL2,BL} as well as Breuil in \cite{BR1}. 

Let us now give a more precise description of our results. Let $E$ be a finite extension of $\Fp$ which is the field of coefficients of all our representations, and let 
\[ \K = \pmat{\Zp^\times & \Zp \\ 0 & \Zp^\times} = \B \cap \KK \]
and let $\Z \simeq \Qp^\times$ be the center of $\B$. If $\sigma_1$ and $\sigma_2$ are two smooth characters of $\Qp^\times$ then $\sigma=\sigma_1 \otimes \sigma_2 : \smat{a & b \\ 0 & d} \to \sigma_1(a) \sigma_2(d)$ is a smooth character of $\K\Z$ and we consider the compactly induced representation $\ind_{\K\Z}^{\B} \sigma$. Note that the Iwasawa decomposition implies that $\B/\K\Z = \G/\KK\Z$ so that $\ind_{\K\Z}^{\B} \sigma$ can be seen as a space of ``twisted functions'' on the tree of $\G$.

\begin{enonce*}{Theorem A}
If $\Pi$ is a smooth irreducible representation of $\B$ admitting a central character, then there exists $\sigma = \sigma_1 \otimes \sigma_2$ such that $\Pi$ is a quotient of $\ind_{\K\Z}^{\B}\sigma$.
\end{enonce*}

This result (theorem \ref{allquot}) is a direct consequence of the fact that a pro-$p$-group acting on a smooth $E$-representation necessarily admits some nontrivial fixed points. Assume now that $\sigma_1(p)=\sigma_2(p)$ and let $\lambda=\sigma_1(p)=\sigma_2(p)$ and let $\mathbf{1}_\sigma$ be the element of $\ind_{\K\Z}^{\B}\sigma$ supported on $\K\Z$ and given there by $\mathbf{1}_\sigma(kz)=\sigma(kz)$. If $n \geq 2$ and if $1 \leq h \leq p^{n-1}-1$, let $S_n(h,\sigma)$ be the subspace of $\ind_{\K\Z}^{\B}\sigma$ generated by the $\B$-translates of 
\[  (-\lambda^{-1})^n \pmat{1 & 0 \\ 0 & p^n} \mathbf{1}_\sigma + \sum_{j=0}^{p^n-1} \binom{j}{h(p-1)} \pmat{1 & -jp^{-n} \\ 0 & 1}   \mathbf{1}_\sigma \] 
and let $\Pi_n(h,\sigma) = \ind_{\K\Z}^{\B}\sigma / S_n(h,\sigma)$. We say that $h$ is primitive if there is no $d<n$ dividing  $n$ such that $h$ is a multiple of $(p^n-1)/(p^d-1)$ (this condition is equivalent to requiring that if we write $h=e_{n-1} \hdots e_1 e_0$ in base $p$, then the map $i \mapsto e_i$ from $\ZZ/n\ZZ$ to $\{0,\hdots,p-1\}$ has no period strictly smaller than $n$). The main result of \S \ref{sirb} is that the $\Pi_n(h,\sigma)$ are irreducible if $h$ is primitive. In chapter \ref{secpgm}, we turn to Galois representations, Fontaine's $(\phi,\Gamma)$-modules and Colmez' $\Omega(\cdot)$ functor. In particular, we give a careful construction of $\Omega(W)$ and in theorem \ref{nultheta}, we prove that there exists a character $\sigma$ such that $\Omega(W)$ is a smooth irreducible quotient of $\ind_{\K\Z}^{\B}\sigma$ by a subspace which contains $S_n(h,\sigma)$ where $n=\dim(W)$ and $h$ depends on $W$. Let $\omega_n$ be Serre's fundamental character of level $n$. For a primitive $1 \leq h \leq p^n-2$, let $\ind(\omega_n^h)$ be the unique representation of $\gal$ whose determinant is $\omega^h$ (where $\omega=\omega_1$ is the mod $p$ cyclotomic character) and whose restriction to the inertia subgroup $\inert$ of $\gal$ is given by  $\omega_n^h \oplus \omega_n^{ph} \oplus \cdots \oplus \omega_n^{p^{n-1}h}$. Every $n$-dimensional absolutely irreducible $E$-linear representation $W$ of $\gal$ is isomorphic to $\ind(\omega_n^h) \otimes \chi$ for some primitive $1 \leq h \leq p^{n-1}-1$ and some character $\chi$ and our main result is then the following (theorem \ref{shlisom}).

\begin{enonce*}{Theorem B}
If $n \geq 2$ and if $1 \leq h \leq p^{n-1}-1$ is primitive, then 
\[ \Omega(\ind(\omega_n^h) \otimes \chi) \simeq \Pi_n(h,\chi\omega^{h-1} \otimes \chi). \]
\end{enonce*}

After that, we give the connection with Breuil's correspondence. Our main result connecting Colmez' functor with Breuil's correspondence is the following (it is a combination of theorem \ref{shlisom} for $n=2$ and theorem \ref{ssgisphl}).

\begin{enonce*}{Theorem C}
If $1 \leq h \leq p-1$, then we have
\[ \Omega(\ind(\omega_2^h) \otimes \chi) \simeq \Pi_2(h,\omega^{h-1} \chi \otimes \chi) \simeq 
\frac{\ind_{\KK\Z}^{\G} \Sym^{h-1} E^2}{T} \otimes (\chi \circ \det), \]
where the last representation is viewed as a representation of $\B$.
\end{enonce*}

It would have been possible to treat the Galois representations of dimension $1$ in the same way, and therefore to get a proof that Colmez' functor gives Breuil's correspondence for reducible representations of dimension $2$ using the methods of this article so that one recovers the corresponding result of \cite{LB5} without using the stereographic projection of \cite{BL,BL2}. I have chosen not to include this as it does not add anything conceptually, but it is an instructive exercise for the reader. 

Finally, if $h=1$ and $n\geq 2$, we can give a more explicit version of theorem B. We define two $\B$-equivariant operators $T_+$ and $T_-$ on $\ind_{\K\Z}^{\B} \sigma$ by
\[ T_+(\mathbf{1}_\sigma) = \sum_{j=0}^{p-1} \pmat{p & j \\ 0 & 1}\mathbf{1}_\sigma
\quad\text{and}\quad 
T_-(\mathbf{1}_\sigma) = \pmat{1 & 0 \\ 0 & p}\mathbf{1}_\sigma \]
so that the Hecke operator is $T=T_+ + T_-$ and theorem B can be restated as follows. 

\begin{enonce*}{Theorem D}
We have
\[ \Omega(\ind(\omega_n) \otimes \chi) \simeq \frac{\ind_{\K\Z}^{\B} (1 \otimes 1)}{T_- + (-1)^n T_+^{n-1}} \otimes (\chi \circ \det). \]
\end{enonce*}

There may be a correspondence between irreducible $E$-linear representations of dimension $n$ of $\gal$ and certain objects coming from $\GL_n(\Qp)$. I hope that theorem C gives a good place to start looking for this correspondence, along with the ideas of \cite{SV}.

\section{Smooth modular representations of $\B_2(\Qp)$}
\label{smr}

In this chapter, we construct a number of representations of $\B$ and show that they are irreducible by reasoning directly on the tree of $\operatorname{PGL}_2(\Qp)$.

\Subsection{Linear algebra over $\Fp$}
\label{linalg}

The binomial coefficients are defined by the formula $(1+X)^n = \sum_{i \in \ZZ} \binom{n}{i} X^i$ and we think of them as living in $\Fp$. The following result is due to Lucas.

\begin{lemm}\label{lucas}
If $a$ and $b$ are integers and $a=a_s \hdots a_0$ and $b=b_s \hdots b_0$ are their expansions in base $p$, then
\[ \binom{a}{b} = \binom{a_s}{b_s} \cdots \binom{a_0}{b_0}. \]
\end{lemm}

\begin{proof}
If we write $(1+X)^a=(1+X)^{a_0} (1+X^p)^{a_1} \cdots (1+X^{p^s})^{a_s}$, then the coefficient of $X^b$ on the left is the coefficient of $X^{b_0}X^{pb_1}\cdots X^{p^s b_s}$ on the right.
\end{proof}

\begin{lemm}\label{prbin}
If $k,\ell \geq 0$ and if
\[ a_{k,\ell} = \sum_{j=0}^{p^n-1} \binom{j}{k}\binom{j}{\ell} \]
then $a_{k,\ell} = 0$ if $k+\ell \leq p^n-2$ and $a_{k,\ell} = (-1)^k$ if $k+\ell = p^n-1$.
\end{lemm}

\begin{proof}
The number $a_{k,\ell}$ is the coefficient of $X^kY^\ell$ in the expansion of \[ \sum_{j=0}^{p^n-1} (1+X)^j(1+Y)^j = \frac{(1+X+Y+XY)^{p^n}-1}{(1+X+Y+XY)-1} = (X+Y+XY)^{p^n-1}. \]
Each term of this polynomial is of the form $X^aY^b(XY)^c$ with $a+b+c=p^n-1$ so that there is no term of total degree $\leq p^n-2$ and the terms of total degree $p^n-1$ are those for which $c=0$ and therefore they are the $(-1)^k X^kY^{p^n-1-k}$.
\end{proof}

Let $V_n$ be the vector space of sequences $(x_0,\hdots,x_{p^n-1})$ with $x_i \in E$. The bilinear map $\pscal{\cdot,\cdot} : V_n \times V_n \to E$ given by $\pscal{x,y} = \sum_{j=0}^{p^n-1} x_j y_j$ is a perfect pairing on $V_n$. 

Let $v_{k,n} \in V_n$ be defined by
\[ v_{k,n} = \left( \binom{0}{k},\binom{1}{k},\hdots,\binom{p^n-1}{k}\right), \]
and let $V_{k,n}$ be the subspace of $V_n$ generated by $v_{0,n},\hdots,v_{k-1,n}$. 

\begin{lemm}\label{orthovkn}
For $0 \leq k \leq p^n$, the space $V_{k,n}$ is of dimension $k$ and $V_{k,n}^\perp = V_{p^n-k,n}$.
\end{lemm}

\begin{proof}
Since the first $j$ components of $v_{j,n}$ are $0$ and the $(j+1)$-th is $1$, the vectors $v_{j,n}$ are linearly independent and $V_{k,n}$ is of dimension $k$. Lemma \ref{prbin} says that $\pscal{v_{j,n},v_{\ell,n}}=0$ if $j + \ell \leq p^n-2$ and this gives us $V_{k,n}^\perp = V_{p^n-k,n}$ by a dimension count.
\end{proof}

In particular, $V_{1,n}$ is the space of constant sequences and $V_{p^n-1,n}$ is the space of zero sum sequences. Note that by lemma \ref{lucas}, we have $\binom{j+p^n}{k}=\binom{j}{k}$ if $0 \leq k \leq p^n-1$ so that we can safely think of the indices of the $x \in V_n$ as belonging to $\ZZ/p^n \ZZ$. Let $\Delta : V_n \to V_n$ be the map defined by $(\Delta x)_j = x_{j-1} - x_j$.

\begin{lemm}\label{del}
If $0 \leq k+\ell \leq p^n$, then $\Delta^k$ gives rise to an exact sequence
\[ 0 \to V_{k,n} \to V_{\ell+k,n} \xrightarrow{\Delta^k}   V_{\ell,n} \to 0, \]
and $\Delta^k(x) \in V_{\ell,n}$ if and only if $x \in V_{\ell+k,n}$. 
\end{lemm}

\begin{proof}
There is nothing to prove if $k=0$ and we now assume that $k=1$. It is clear that $\ker(\Delta)=V_{1,n}$ the space of constant sequences, and the formula
\[ \binom{j}{m} - \binom{j-1}{m} = \binom{j-1}{m-1} \]
implies that $\Delta(V_{\ell+1,n}) \subset V_{\ell,n}$ so that by counting dimensions we see that there is indeed an exact sequence $0 \to V_{1,n} \to V_{\ell+1,n} \xrightarrow{\Delta}   V_{\ell,n} \to 0$. If $\Delta(x) \in V_{\ell,n}$ then this implies that there exists $y \in V_{\ell+1,n}$ such that $\Delta(x)=\Delta(y)$ so that $x \in V_{\ell+1,n} + \ker(\Delta) = V_{\ell+1,n}$. This proves the lemma for $k=1$ and for $k \geq 2$, it follows from a straightforward induction.
\end{proof}

Note that $\Delta$ is nilpotent of rank $p^n$ and therefore the only subspaces of $V_n$ stable under $\Delta$ are the $\ker (\Delta^k) = V_{k,n}$. Since the cyclic shift $(x_j) \mapsto (x_{j-1})$ is equal to $\Id+\Delta$, this also implies that the only subspaces of $V_n$ stable under the cyclic shift are the $V_{k,n}$.

If $a \in \Zp$ then let $\mu_a : V_n \to V_n$ be the map defined by $\mu_a (x)_j = x_{aj}$.

\begin{lemm}\label{mua}
We have $\mu_a(v_{k,n}) - a^k v_{k,n} \in V_{k,n}$ so that if $x \in V_{k+1,n}$ then $\mu_a (x) \in V_{k+1,n}$.
\end{lemm}

\begin{proof}
We prove both claims by induction, assuming that it is true for $\ell \leq k-1$ (it is immediate if $\ell=0$ or even $\ell=1$). Vandermonde's identity gives us
\[ \binom{aj}{k} = \binom{aj-a}{k} \binom{a}{0} + \binom{aj-a}{k-1} \binom{a}{1} + \cdots + \binom{aj-a}{0} \binom{a}{k}, \]
which shows that $\Delta \circ \mu_a(v_{k,n}) -a^k  v_{k-1,n} \in V_{k-1,n}$ by the induction hypothesis and therefore that $\mu_a(v_{k,n}) - a^k  v_{k,n} \in  V_{k,n}$ by lemma \ref{del} which finishes the induction. 
\end{proof}

\begin{lemm}\label{extrbl}
If $x \in V_{k,n}$ and if $0 \leq i \leq p-1$, then the sequence $y \in V_{n-1}$ given by $y_j=x_{pj+i}$ belongs to $V_{\lfloor (k-1)/p \rfloor+1, n-1}$.
\end{lemm}

\begin{proof}
If $\ell \leq k-1$ and if we write $\ell =p\lfloor \ell/p \rfloor+\ell_0$ so that $0 \leq \ell_0 \leq p-1$, then by lemma \ref{lucas}, we have
\[ \binom{pj+i}{\ell} = \binom{j}{\lfloor \ell/p \rfloor}\binom{i}{\ell_0}, \]
which implies the lemma.
\end{proof}

\Subsection{The twisted tree}\label{twtr}
We now turn to $\B/\K\Z$ and the smooth representations of $\B$. 
If $\beta \in \Qp$ and $\delta \in \ZZ$, let
\[ g_{\beta,\delta} = \pmat{ 1 & \beta \\ 0 & p^\delta}. \]
Let $A=\{ \alpha_n p^{-n} + \cdots + \alpha_1 p^{-1}$ where $0 \leq \alpha_j \leq p-1 \}$ so that $A$ is a system of representatives of $\Qp/\Zp$.

\begin{lemm}\label{cosets}
We have $\B = \coprod_{\beta \in A, \delta\in \ZZ} g_{\beta,\delta} \cdot \K\Z$.
\end{lemm}

\begin{proof}
If $\smat{a & b \\ 0 & d} \in \B$, then with obvious notations we have
\[ \pmat{a & b \\ 0 & d} = \pmat{a_0 p^\alpha & b \\ 0 & d_0 p^\delta} = \pmat{1 & b p^{-\alpha} d_0^{-1} - c  \\ 0 & p^{\delta-\alpha}} \pmat{a_0  & cd_0 \\ 0 & d_0 } \pmat{p^\alpha & 0 \\ 0 & p^\alpha} \]
which tells us that $\B = \cup_{\beta \in A, \delta\in \ZZ} g_{\beta,\delta} \cdot \K\Z$ since we can always choose $c \in \Zp$ such that $bp^{-\alpha}d_0^{-1}-c \in A$. The fact that the union is disjoint is immediate.
\end{proof}

The vertices of the tree of $\G$ can then be labelled by the $\delta \in \ZZ$ and the $\beta \in A$.

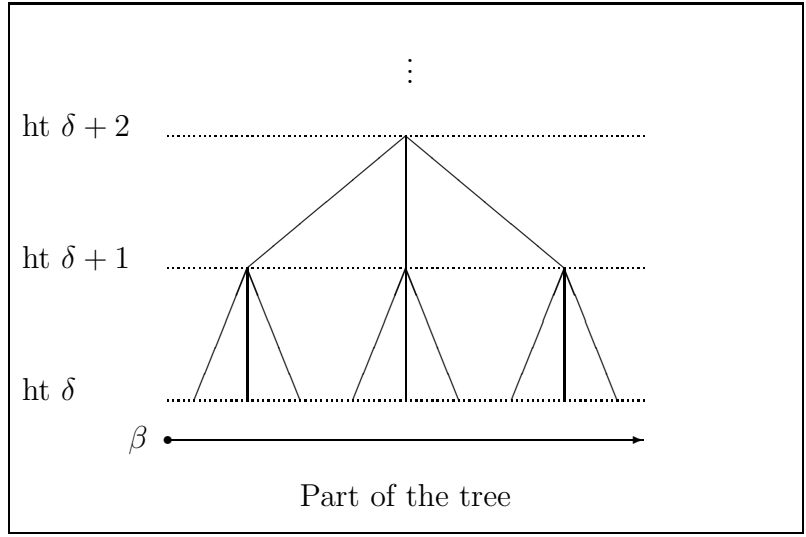
\begin{center}
\begin{picture}(300,200)
\put(0,0){\framebox(300,200){}}
\put(5,50){ht $\delta$}
\put(5,100){ht $\delta+1$}
\put(5,150){ht $\delta+2$}

\dottedline{2}(60,50)(240,50)
\dottedline{2}(60,100)(240,100)
\dottedline{2}(60,150)(240,150)

\put(70,50){\line(2,5){20}}
\put(90,50){\line(0,5){50}}
\put(110,50){\line(-2,5){20}}

\put(130,50){\line(2,5){20}}
\put(150,50){\line(0,5){50}}
\put(170,50){\line(-2,5){20}}

\put(190,50){\line(2,5){20}}
\put(210,50){\line(0,5){50}}
\put(230,50){\line(-2,5){20}}

\put(90,100){\line(6,5){60}}
\put(150,100){\line(0,5){50}}
\put(210,100){\line(-6,5){60}}

\put(150,170){$\vdots$}

\put(45,32){$\beta$}
\put(60,35){\circle*{3}}
\put(60,35){\vector(1,0){180}}

\put(110,10){Part of the tree}
\end{picture}
\end{center}

If $\sigma_1$ and $\sigma_2$ are two smooth characters $\sigma_i : \Qp^\times \to E^\times$, then let $\sigma = \sigma_1 \otimes \sigma_2 : \K\Z \to E^\times$ be the character $\sigma : \smat{a & b \\ 0 & d} \mapsto \sigma_1(a)\sigma_2(d)$ and let $\ind_{\K\Z}^{\B}\sigma$ be the set of functions $f : \B \to E$ satisfying $f(kg) = \sigma(k)f(g)$ if $k \in \K\Z$ and such that $f$ has compact support modulo $\Z$. If $g \in \B$, denote by $[g]$ the function $[g] : \B \to E$ defined by $[g](h)=\sigma(hg)$ if $h \in \K\Z g^{-1}$ and $[g](h)=0$ otherwise. Every element of $\ind_{\K\Z}^{\B}\sigma$ is a finite linear combination of some functions $[g]$. We make $\ind_{\K\Z}^{\B}\sigma$ into a representation of $\B$ in the usual way: if $g \in \B$, then $(gf)(h)=f(hg)$. In particular, we have $g [h] = [gh]$ in addition to the formula $[gk]=\sigma(k)[g]$ for $k\in\K\Z$.

\begin{lemm}\label{intert}
If $\chi$ is a smooth character of $\Qp^\times$ then the map $[g] \mapsto (\chi \circ \det)(g)^{-1} [g]$ extends to a $\B$-equivariant isomorphism from $(\ind_{\K\Z}^{\B}\sigma) \otimes (\chi \circ \det)$ to $\ind_{\K\Z}^{\B}(\sigma_1 \chi \otimes \sigma_2 \chi)$.
\end{lemm}

\begin{proof}
Let us write $[\cdot]_\sigma$ and $[\cdot]_{\sigma\chi}$ for the two functions $[\cdot]$ in the two induced representations. We then have $h [g]_\sigma = (\chi \circ \det)(h) [hg]_{\sigma}$ and 
\[ (\chi \circ \det)(g)^{-1} h [g]_{\sigma\chi} = (\chi \circ \det)(h) (\chi \circ \det)(hg)^{-1} [hg]_{\sigma\chi} \] 
so that the above map is indeed $\B$-equivariant.
\end{proof}

Each $f \in \ind_{\K\Z}^{\B}\sigma$ can be written in a unique way as $f = \sum_{\beta,\delta} \alpha(\beta,\delta) [g_{\beta,\delta}]$. The formula
\[ \pmat{ 1 & \beta+\lambda \\ 0 & p^\delta} = \pmat{ 1 & \beta \\ 0 & p^\delta} \pmat{ 1 & \lambda \\ 0 & 1} \]
and the fact that $\sigma$ is trivial on $\smat{ 1 & \Zp \\ 0 & 1}$ imply that we can extend the definition of $\alpha(\beta,\delta)$ to all $\beta \in \Qp$. We then have the formula $\alpha(\beta,\delta)\left(\smat{1 & \lambda \\ 0 & 1}f\right) = \alpha(\beta-\lambda p^\delta,\delta)(f)$ if $\lambda \in \Qp$.

The \emph{support} of $f$ is the set of $g_{\beta,\delta}$ such that $\alpha(\beta,\delta) \neq 0$. Let us say that the \emph{height} of an element $g_{\beta,\delta}$ is $\delta$. We say that $f \in \ind_{\K\Z}^{\B}\sigma$ has support in \emph{levels} $n_1,\hdots,n_k$ if all the elements of its support are of height $n_i$ for some $i$. If $f \in \ind_{\K\Z}^{\B}\sigma$, then we can either raise or lower the support of $f$ using the formula $\smat{1 & 0 \\ 0 & p^{\pm 1}} g_{\beta,\delta} = g_{\beta,\delta \pm 1}$. 

If $n \geq 0$ let us say that an \emph{$n$-block} is the set of $g_{\beta - j p^{-n},\delta}$ for $j=0,\hdots,p^n-1$ and that the \emph{initial} $n$-block is the one for which $\beta=0$. We use the same name for the vector of coefficients $\alpha(\beta - j p^{-n},\delta)$ for $j=0,\hdots,p^n-1$ so that an $n$-block is then an element of $V_n$ from \S \ref{linalg}.

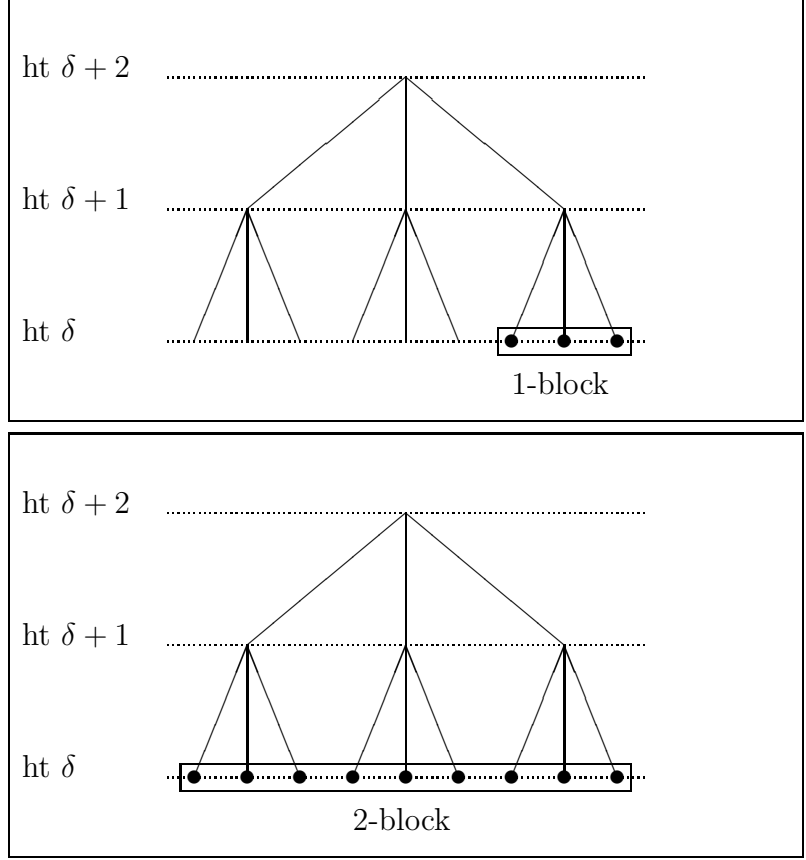
\begin{center}
\begin{picture}(300,160)(0,20)
\put(0,20){\framebox(300,160){}}
\put(5,50){ht $\delta$}
\put(5,100){ht $\delta+1$}
\put(5,150){ht $\delta+2$}

\dottedline{2}(60,50)(240,50)
\dottedline{2}(60,100)(240,100)
\dottedline{2}(60,150)(240,150)

\put(70,50){\line(2,5){20}}
\put(90,50){\line(0,5){50}}
\put(110,50){\line(-2,5){20}}

\put(130,50){\line(2,5){20}}
\put(150,50){\line(0,5){50}}
\put(170,50){\line(-2,5){20}}

\put(190,50){\line(2,5){20}}
\put(210,50){\line(0,5){50}}
\put(230,50){\line(-2,5){20}}

\put(90,100){\line(6,5){60}}
\put(150,100){\line(0,5){50}}
\put(210,100){\line(-6,5){60}}

\put(190,50){\circle*{5}}
\put(210,50){\circle*{5}}
\put(230,50){\circle*{5}}

\put(185,45){\framebox(50,10){}}
\put(190,30){$1$-block}

\end{picture}
\end{center}

\vspace{\baselineskip}

\begin{center}
\begin{picture}(300,160)
\put(0,20){\framebox(300,160){}}
\put(5,50){ht $\delta$}
\put(5,100){ht $\delta+1$}
\put(5,150){ht $\delta+2$}

\dottedline{2}(60,50)(240,50)
\dottedline{2}(60,100)(240,100)
\dottedline{2}(60,150)(240,150)

\put(70,50){\line(2,5){20}}
\put(90,50){\line(0,5){50}}
\put(110,50){\line(-2,5){20}}

\put(130,50){\line(2,5){20}}
\put(150,50){\line(0,5){50}}
\put(170,50){\line(-2,5){20}}

\put(190,50){\line(2,5){20}}
\put(210,50){\line(0,5){50}}
\put(230,50){\line(-2,5){20}}

\put(90,100){\line(6,5){60}}
\put(150,100){\line(0,5){50}}
\put(210,100){\line(-6,5){60}}

\put(190,50){\circle*{5}}
\put(210,50){\circle*{5}}
\put(230,50){\circle*{5}}

\put(130,50){\circle*{5}}
\put(150,50){\circle*{5}}
\put(170,50){\circle*{5}}

\put(70,50){\circle*{5}}
\put(90,50){\circle*{5}}
\put(110,50){\circle*{5}}

\put(65,45){\framebox(170,10){}}
\put(130,30){$2$-block}

\end{picture}
\end{center}

In the following paragraph, we study some irreducible quotients of $\ind_{\K\Z}^{\B}\sigma$ of arithmetic interest but before we do that, it is worthwhile to point out that all smooth irreducible representations of $\B$ admitting a central character are a quotient of some $\ind_{\K\Z}^{\B}\sigma$.

\begin{theo}\label{allquot}
If $\Pi$ is a smooth irreducible representation of $\B$ admitting a central character, then there exists $\sigma = \sigma_1 \otimes \sigma_2$ such that $\Pi$ is a quotient of $\ind_{\K\Z}^{\B}\sigma$.
\end{theo}

\begin{proof}
The group $\I$ defined by
\[ \I = \pmat{1+p\Zp & \Zp \\ 0 & 1+p\Zp} \]
is a pro-$p$-group and hence $\Pi^{\I} \neq 0$. Furthermore, $\I$ is a normal subgroup of $\K$ so that $\Pi^{\I}$ is a representation of $\K/\I = \Fp^\times \times \Fp^\times$. Since this group is a finite group of order prime to $p$, we have $\Pi^{\I} = \oplus_{\eta} \Pi^{\K=\eta}$ where $\eta$ runs over the characters of $\Fp^\times \times \Fp^\times$ and since $\Z$ acts through a character by hypothesis, there exists a character $\sigma = \sigma_1 \otimes \sigma_2$ of $\K\Z$ and $v \in \Pi$ such that $k \cdot v = \sigma(k)v$ for $k\in\K\Z$. By Frobenius reciprocity, we get a nontrivial map $\ind_{\K\Z}^{\B}\sigma \to \Pi$ and this map is surjective since $\Pi$ is irreducible.
\end{proof}

Note that $\sigma$ is not uniquely determined by $\Pi$: there are  nontrivial intertwinings between some quotients of $\ind_{\K\Z}^{\B}\sigma$ for different $\sigma$. 

We finish this paragraph with a useful general lemma. Let $\tau_k = \smat{1 & -1/p^k \\ 0 & 1}$ and let $\Pi$ be any representation of $\B$.

\begin{lemm}\label{existfix}
If $v \neq 0 \in \Pi^{\smat{1 & \Zp \\ 0 & 1}}$ and if $k \geq 0$ then one of the $p^k$ elements
\[ v_\ell = \sum_{j=0}^{p^k-1} \binom{j}{\ell} \tau_k^j (v), \quad 0 \leq \ell \leq p^k-1 \]
is nonzero and fixed by $\tau_k$.
\end{lemm}

\begin{proof}
If all $p^k$ elements above were zero then lemma \ref{orthovkn} would imply that for any sequence $x = (x_j) \in V_k$ we would have $\sum_{j=0}^{p^k-1} x_j \tau_k^j  (v)= 0$ and with $x=(1,0,\hdots,0)$, we get $v=0$. Let $\ell$ be the smallest integer such that $v_\ell \neq 0$. If $\ell=0$ then $\tau_k(v_0)-v_0 = 0$ since $\tau_k^{p^k} = \tau_0 \in \smat{1 & \Zp \\ 0 & 1}$  and otherwise $\tau_k (v_\ell) - v_\ell = - v_{\ell-1} =0$.
\end{proof}

\Subsection{Some irreducible representations of $\B_2(\Qp)$}
\label{sirb}

If $n \geq 1$ and $0 \leq \ell \leq p^n-1$, let $w_{\ell,n} \in \ind_{\K\Z}^{\B}\sigma$ be the element 
\[ w_{\ell,n} = \sum_{j=0}^{p^n-1} \binom{j}{\ell} \left[\pmat{1 & -jp^{-n} \\ 0 & 1} \right], \]
so that the initial $n$-block of $w_{\ell,n}$ is $v_{\ell,n}$.

\begin{defi}\label{shl}
If $n \geq 2$ and if $1 \leq h \leq p^{n-1}-1$ and if $\sigma = \sigma_1 \otimes \sigma_2$ is a character of $\K\Z$ such that $\sigma_1(p)=\sigma_2(p)$, let $\lambda = \sigma_1(p)=\sigma_2(p)$ and let $S_n(h,\sigma)$ be the subspace of $\ind_{\K\Z}^{\B}\sigma$ generated by the translates under the action of $\B$ of $(-\lambda^{-1})^n \left[\smat{1 & 0 \\ 0 & p^n}\right] + w_{h(p-1),n}$.
\end{defi}

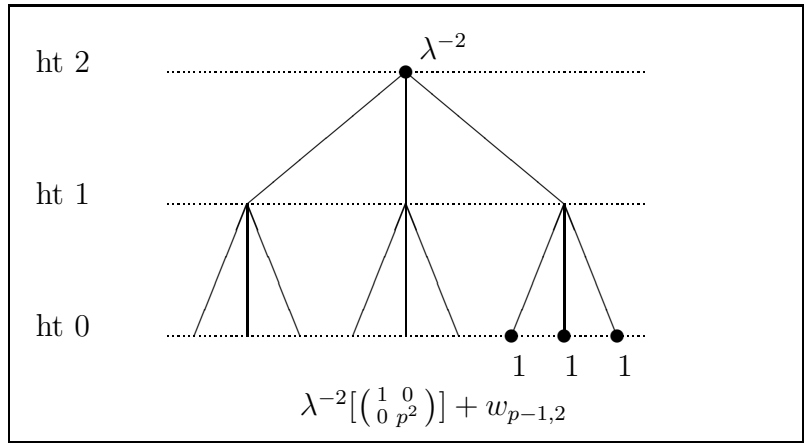
\begin{center}
\begin{picture}(300,165)(0,10)

\put(0,10){\framebox(300,165){}}
\put(10,50){ht $0$}
\put(10,100){ht $1$}
\put(10,150){ht $2$}

\dottedline{2}(60,50)(240,50)
\dottedline{2}(60,100)(240,100)
\dottedline{2}(60,150)(240,150)

\put(70,50){\line(2,5){20}}
\put(90,50){\line(0,5){50}}
\put(110,50){\line(-2,5){20}}

\put(130,50){\line(2,5){20}}
\put(150,50){\line(0,5){50}}
\put(170,50){\line(-2,5){20}}

\put(190,50){\line(2,5){20}}
\put(210,50){\line(0,5){50}}
\put(230,50){\line(-2,5){20}}

\put(90,100){\line(6,5){60}}
\put(150,100){\line(0,5){50}}
\put(210,100){\line(-6,5){60}}

\put(150,150){\circle*{5}}
\put(190,50){\circle*{5}}
\put(210,50){\circle*{5}}
\put(230,50){\circle*{5}}

\put(155,155){$\lambda^{-2}$}
\put(190,35){1}
\put(210,35){1}
\put(230,35){1}

\put(110,20){$\lambda^{-2} [\smat{1 & 0 \\ 0 & p^2}] + w_{p-1,2}$}

\end{picture}
\end{center}

The representations we are interested in are the quotients $\Pi_n(h,\sigma) = \ind_{\K\Z}^{\B} \sigma / S_n(h,\sigma)$ and the main result of this chapter is that they are irreducible if $h$ is primitive. Before we can prove this, we need a number of technical results. 

If $f \in  \ind_{\K\Z}^{\B} \sigma$ and if $0 \leq i \leq n-1$, let 
\[ f_i = \sum_{\substack{\beta \in A \\ \delta \equiv i \bmod{n}}} 
\alpha(\beta,\delta) [g_{\beta,\delta}] \] 
so that $f=f_0+f_1+\cdots+f_{n-1}$.

\begin{lemm}\label{evenodd}
If $f \in  \ind_{\K\Z}^{\B} \sigma$ then $f \in S_n(h,\sigma)$ if and only if $f_i \in S_n(h,\sigma)$ for all $0 \leq i \leq n-1$.
\end{lemm}

\begin{proof}
We need only to check that if $f \in  S_n(h,\sigma)$ then $f_i \in S_n(h,\sigma)$ and this follows from the fact that $S_n(h,\sigma)$ is generated by elements which have their supports in levels equal modulo $n$. 
\end{proof}

Let $i_{n-1} \hdots i_1 i_0$ be the expansion of $h(p-1)$ in base $p$. Note that $h \leq p^{n-1}-1$ implies that $i_{n-1} \leq p-2$. Let $h_k = i_{n-k}  + p i_{n-k+1} + \cdots + p^{k-1} i_{n-1}$ so that $h_k = ph_{k-1} + i_{n-k}$ and $h_0=0$ and $h_n=h(p-1)$. Recall that the vectors $v_{k,n}$ were defined in \S \ref{linalg} and let $\B^+ = \coprod_{\beta \in A, \delta \geq 0} g_{\beta,\delta} \K\Z$.

\begin{lemm}\label{oneshtb}
If the support of $g \in S_n(h,\sigma)$ is in levels $\geq 0$, then
\begin{enumerate}
\item $g$ is a linear combination of $\B^+$-translates of $(-\lambda^{-1})^n \left[\smat{1 & 0 \\ 0 & p^n}\right] + w_{h(p-1),n}$
\item if $1 \leq k \leq n$, then the $k$-blocks of level $0$ of $g$ are in $V_{h_k+1,k}$.
\end{enumerate}
\end{lemm}

\begin{proof}
Note first that if $\smat{a & c \\ 0 & d} \in \K\Z$, then
\begin{align*} 
\pmat{a & c \\ 0 & d} w_{\ell,n} & = \sum_{j=0}^{p^n-1} \binom{j}{\ell} \pmat{a & c \\ 0 & d} \left[\pmat{1 & -jp^{-n} \\ 0 & 1} \right] \\
& = \sum_{j=0}^{p^n-1} \binom{j}{\ell}  \left[\pmat{1 & -jp^{-n}ad^{-1} \\ 0 & 1}  \pmat{a & c \\ 0 & d} \right] \\
& = \sigma_1(a)\sigma_2(d) \sum_{j=0}^{p^n-1} \binom{jda^{-1}}{\ell}  \left[\pmat{1 & -jp^{-n} \\ 0 & 1} \right],
\end{align*}
and note also that the initial $n$-block of $\smat{1 & jp^{-n} \\ 0 & 1} w_{\ell,n} -w_{\ell,n}$ is in $V_{\ell,n}$. 

Let us now prove (1). Set $\B^0=\{ \smat{a & b \\  0 & d} \in \B$ such that $\val_p(a)=\val_p(d)\}$. It is enough to prove that any $\B^0$-linear combination of $\phi = (-\lambda^{-1})^n \left[\smat{1 & 0 \\ 0 & p^n}\right] + w_{h(p-1),n}$ which is zero in level $0$ is actually identically zero. If $\sum_{i \in I} \lambda_i \smat{a_i & b_i \\ 0 & d_i} \cdot \phi$ is such a combination where we assume for example (using the action of the center) that $d_i=1$, then the terms indexed by $i_1$ and $i_2$ contribute to the same $n$-block in level $0$ if and only if $b_{i_1} - b_{i_2} \in p^{-n}\Zp$ and we can therefore assume that 
\[ \pmat{a_i & b_i \\ 0 & d_i} \in S = \pmat{\Zp^\times & p^{-n}\Zp \\ 0 & \Zp^\times} \]
so that we're looking at the initial $n$-block. The formulas above and lemma \ref{mua} applied to $da^{-1} \in \Zp^\times$ show that if $g = \smat{a & c \\ 0 & d} \in S$, then the initial $n$-block of $g \cdot \phi - \sigma_1(a)\sigma_2(d) \phi$ belongs to $V_{h(p-1),n}$ so that in a linear combination of $S$-translates of $\phi$, the coefficient of $\left[\smat{1 & 0 \\ 0 & p^n}\right]$ is a nonzero multiple of the coefficient of $w_{h(p-1),n}$; if the latter is zero, then so is the former and our linear combination is identically zero. 

Let us now prove (2). The conclusion of (2) is stable under linear combinations of $\B^+$-translates so by (1) we only need to check that if $b \in \B^+$ then the $k$-blocks of $b w_{h_n,n}$ are in $V_{h_k+1,k}$. If $b=\Id$ then the $n$-block of $w_{h_n,n}$ is $v_{h_n,n}$ which belongs to $V_{h_n+1,n}$ by definition. If we know that the $k$-blocks are in $V_{h_k+1,k}$ then the fact that $\lfloor h_k / p \rfloor = h_{k-1}$ and lemma \ref{extrbl} imply that the $(k-1)$-blocks are in $V_{h_{k-1}+1,k-1}$ so we are done by induction. Next, the above formula for $\smat{a & c \\ 0 & d} w_{\ell,n}$ and lemma \ref{mua} applied to $da^{-1} \in \Zp^\times$ show that the $n$-blocks of the $\smat{a & c \\ 0 & d} w_{h_n,n}$ are contained in $V_{h_n+1,n}$ and we are reduced to the claim above. Finally, $g_{\beta,\delta} \cdot f$ is $f$ moved up by $\delta$ and shifted by $\beta$ and the conclusion of (2) is unchanged under those two operations since the $V_{k,n}$ are stable under the cyclic shift. 
\end{proof}

Recall that $\tau_k = \smat{1 & -1/p^k \\ 0 & 1}$ and that $\alpha(\beta,\delta)(\tau_k(f)) = \alpha(\beta+p^{\delta-k},\delta)(f)$ so that the effect of $\tau_k-\Id$ on a $k$-block $y$ in level $0$ is to replace it with $\Delta(y)$.

\begin{lemm}\label{freeblock}
If the support of $f \in S_n(h,\sigma)$ is contained in a single $k$-block with $0 \leq k \leq n$, then this $k$-block is in $V_{h_k,k}$ and all such elements do occur : $w_{\ell,k} \in S_n(h,\sigma)$ for $0 \leq \ell \leq h_k-1$.
\end{lemm}

\begin{proof}
If $k=n$ then the $n$-block of $\tau_n(w_{h_n,n})-w_{h_n,n}$ is $v_{h_n-1,n}$ and the set of possible $n$-blocks is stable under the cyclic shift so we get all of $V_{h_n,n}$ but not $V_{h_n+1,n}$ since $\Pi_n(h,\sigma) \neq 0$. If some $v_{\ell,k}$ occurs as the $k$-block of some $f$, wlog in level $0$, then for all $0 \leq m \leq p-1$ the $(k+1)$-block of $\sum_{i=0}^{p-1} \binom{i}{m} \tau_{k+1}^i(f)$ is $[\binom{0}{m}v_{\ell,k},\hdots,\binom{p-1}{m}v_{\ell,k}]$ and this is $v_{p\ell+m,k+1}$ since $\binom{j}{\ell}\binom{i}{m}=\binom{pj+i}{p\ell+m}$ by lemma \ref{lucas}. In particular if $v_{h_k,k}$ occurred then so would $v_{h_{k+1},k+1}$ and we get a contradiction. Conversely, assuming inductively that the second assertion of the lemma holds for $k+1$, this tells us that all $[\binom{0}{m}v_{\ell,k},\hdots,\binom{p-1}{m}v_{\ell,k}]$ occur as a $(k+1)$-block for $p\ell + m \leq h_{k+1}-1$ and by taking $m=p-1$ and $\ell \leq h_k-1$ we obtain $v_{\ell,k}$ and we are done by a descending induction on $k$.
\end{proof}

Let us write as above a $(n+1)$-block as $[b_0,\dots,b_{p-1}]$ where each $b_i$ is a $n$-block. 

\begin{lemm}\label{shtb}
If the support of $g \in S_n(h,\sigma)$ is in levels $0$, $1$, \dots, $n-1$ then the $(n+1)$-blocks of level $0$ of $g$ are of the form
\[ [\mu_0 v_{h_n,n}+x_0,\dots,\mu_{p-1} v_{h_n,n}+x_{p-1}], \] 
where $x_i \in V_{h_n,n}$ and $(\mu_0,\hdots,\mu_{p-1}) \in V_{h_1+1,1}$.
\end{lemm}

\begin{proof}
By lemma \ref{evenodd}, we may assume that the support of $g$ is in level $0$ and lemma \ref{oneshtb} tells us that the $n$-blocks of $g$ are in $V_{h_n+1,n}$ so that each of them can be written as $\mu_i v_{h_n,n}+x_i$ where $x_i \in V_{h_n,n}$. By subtracting from $g$ appropriate combinations of translates of the $w_{\ell,n}$ with $0 \leq \ell \leq h_n-1$ we get a $g'$ such that $x_i=0$ for all $i$ and by subtracting appropriate combinations of translates of $(-\lambda^{-1})^n \left[\smat{1 & 0 \\ 0 & p^n}\right] + w_{h_n,n}$ from $g'$ we get an element $g''$ of $S_n(h,\sigma)$ with support in level $n$ and whose $1$-blocks are the $-(-\lambda^{-1})^n (\mu_0,\hdots,\mu_{p-1})$. Lemma \ref{oneshtb} applied to $\smat{1 & 0 \\ 0 & 1/p^n} g''$ gives us $(\mu_0,\hdots,\mu_{p-1}) \in V_{h_1+1,1}$.
\end{proof}

\begin{coro}\label{tauinvh}
If the support of $f \in  \ind_{\K\Z}^{\B} \sigma$ is in levels $0$, $1$, \dots, $n-1$ and if $\tau_{n+1}(f) - f \in S_n(h,\sigma)$, then the $n$-blocks of level $0$ of $f$ are in $V_{h_n+1,n}$.
\end{coro}

\begin{proof}
Lemma \ref{shtb} applied to $\tau_{n+1}(f)-f$ tells us that the $(n+1)$-blocks of $\tau_{n+1}(f)-f$ in level $0$ are of the form $[\mu_0 v_{h_n,n}+x_0,\dots,\mu_{p-1} v_{h_n,n}+x_{p-1}]$ with $x_i \in V_{h_n,n}$ and $(\mu_0,\hdots,\mu_{p-1}) \in V_{h_1+1,1}$. If we write $f = \sum_{\beta,\delta} \alpha(\beta,\delta) [g_{\beta,\delta}]$, then the coefficient of $[g_{\beta,0}]$ in $\tau_{n+1}(f)-f$ is $\alpha(\beta+p^{-n-1},0) - \alpha(\beta,0)$ so that the $n$-blocks of $\tau_{n+1}(f)-f$ are given by (for readability, we omit both $\beta$ and $\delta=0$ from the notation)

{\footnotesize{\[ \begin{matrix}
\alpha\left(\frac{1}{p^{n+1}}\right) - \alpha(0) &  \alpha\left(\frac{1}{p^{n+1}}+\frac{1}{p^n}\right) - \alpha\left(\frac{1}{p^n}\right) & \hdots & \alpha\left(\frac{1}{p^{n+1}}+\frac{p^n-1}{p^n}\right) - \alpha\left(\frac{p^n-1}{p^n} \right)   \\
\alpha\left(\frac{2}{p^{n+1}}\right) - \alpha\left(\frac{1}{p^{n+1}}\right) & \alpha\left(\frac{2}{p^{n+1}}+\frac{1}{p^n}\right) - \alpha\left(\frac{1}{p^{n+1}}+\frac{1}{p^n}\right) & \hdots & \alpha\left(\frac{2}{p^{n+1}}+\frac{p^n-1}{p^n}\right) - \alpha\left(\frac{1}{p^{n+1}}+\frac{p^n-1}{p^n}  \right) \\
\vdots & \vdots & & \vdots \\
\alpha\left(\frac{p}{p^{n+1}}\right) - \alpha\left(\frac{p-1}{p^{n+1}}\right) & \alpha\left(\frac{p}{p^{n+1}}+\frac{1}{p^n}\right) - \alpha\left(\frac{p-1}{p^{n+1}}+\frac{1}{p^n}\right)  & \hdots & \alpha\left(\frac{p}{p^{n+1}}+\frac{p^n-1}{p^n}\right) - \alpha\left(\frac{p-1}{p^{n+1}}+\frac{p^n-1}{p^n}  \right).
\end{matrix} \]}}

Let $y_0,\hdots,y_{p-1}$ be the $n$-blocks of the $(n+1)$-block of $f$ we are considering. By summing the rows of the above array, we get (recall that $\alpha(\beta)=\alpha(1+\beta)$)
\[ \begin{matrix}  \alpha\left(\beta+\frac{1}{p^n}\right) - \alpha(\beta) &
\alpha\left(\beta+\frac{2}{p^n}\right) - \alpha\left(\beta+\frac{1}{p^n}\right) &
\hdots & \alpha(\beta) - \alpha\left(\beta+\frac{p^n-1}{p^n}\right) 
\end{matrix} \]
which is $\Delta(y_0)$ so that
\[ \Delta(y_0) = \sum_{i=0}^{p-1} (\mu_i v_{h_n,n}+x_i) 
=  \sum_{i=0}^{p-1} x_i \in  V_{h_n,n} \] 
since $\sum_{i=0}^{p-1} \mu_i =0$ because $(\mu_0,\hdots,\mu_{p-1}) \in V_{h_1+1,1}$ with $h_1 +1 = i_{n-1} + 1 \leq p-1$ and if $\Delta(y_0) \in V_{h_n,n}$ then $y_0 \in V_{h_n+1,n}$ by lemma \ref{del}. The same result holds for $y_j$ by applying the previous reasoning to $\tau_{n+1}^j(f)$.
\end{proof}

\begin{coro}\label{furtherwid}
If the support of $f \in \ind_{\K\Z}^{\B} \sigma$ is in levels $0$, $1$, \dots, $n-1$ and the support in level $0$ is included in a single $n$-block and $\tau_n(f)-f  \in S_n(h,\sigma)$, then the $n$-block of $f$ in level $0$ is in $V_{h_n+1,n}$.
\end{coro}

\begin{proof}
Lemma \ref{shtb} applied to $g = \tau_n(f)-f$ tells us that the $n$-block of $\tau_n(f)-f$ is of the form $\mu_0 v_{h_n,n}+x_0$ with $x_0 \in V_{h_n,n}$ and $(\mu_0,0,\hdots,0) \in V_{h_1+1,1}$ so that $\mu_0=0$ since $h_1 +1 \leq p-1$. If $y$ denotes the $n$-block of $f$ then the $n$-block of $\tau_n(f)-f$ is $\Delta(y)$ so that $\Delta(y) \in V_{h_n,n}$ and therefore $y \in V_{h_n+1,n}$ by lemma \ref{del}.
\end{proof}

If $n \geq 1$ and if $1 \leq h \leq p^n-2$, we say that $h$ is primitive if there is no $d<n$ dividing $n$ such that $h$ is a multiple of $(p^n-1)/(p^d-1)$. This condition is equivalent to requiring that if we write $h=e_{n-1} \hdots e_1 e_0$ in base $p$, then the map $i \mapsto e_i$ from $\ZZ/n\ZZ$ to $\{0,\hdots,p-1\}$ has no period strictly smaller than $n$.

\begin{theo}\label{irredim2}
If $n \geq 2$ and if $1 \leq h \leq p^{n-1}-1$ is primitive, then $\Pi_n(h,\sigma)$ is irreducible.
\end{theo}

\begin{proof}
It is enough to show that if $f \in  \ind_{\K\Z}^{\B} \sigma$ is such that $\overline{f} \neq 0$ in $\Pi_n(h,\sigma)$ then some linear combination of translates of $f$ is equal to $[\Id] \bmod{S_n(h,\sigma)}$. 

Suppose that the support of $f$ is in levels $\geq a$. Since $(-\lambda^{-1})^n \left[\smat{1 & 0 \\ 0 & p^n}\right] + w_{h_n,n}$ is an element whose support is one element of height $n$ and a $n$-block of height $0$, by subtracting suitable linear combinations of translates of this from $f$ we may assume that the support of $f$ is in levels $a$, $a+1$, \dots, $a+n-1$; multiplying $f$ by some power of $\smat{1 & 0 \\ 0 & p}$ we may then assume that the support of $f$ is in levels $0$, $1$, \dots, $n-1$. In particular, we have $f \in (\ind_{\K\Z}^{\B} \sigma)^{\smat{1 & \Zp \\ 0 & 1}}$. Let  $s_0,s_1,\hdots,s_{n-1} \gg 0$ be such that the support of $f$ is included in the initial $s_0$-block in level $0$,  the initial $s_1$-block in level $1$, \dots, the initial $s_{n-1}$-block in level $n-1$.

Lemma \ref{existfix} applied with $k=n+1$ shows that we may replace $f$ by one of the $\sum_{j=0}^{p^{n+1}-1} \binom{j}{\ell} \tau_{n+1}^j (f)$ so that $\tau_{n+1}(f) - f \in S_n(h,\sigma)$. The support of this new $f$ is included in the initial $\max(s_j,n+1-j)$-block in level $j$ for $0 \leq j \leq n-1$. Corollary \ref{tauinvh} then shows that there exists $g \in S_n(h,\sigma)$ which is a linear combination of $\smat{1 & \Qp \\ 0 & 1}$-translates of $(-\lambda^{-1})^n \left[\smat{1 & 0 \\ 0 & p^n}\right] + w_{h_n,n}$ and of the $w_{\ell,n}$ for $0 \leq \ell \leq h_n-1$ such that the $n$-blocks of $f$ in level $0$ are the same as the $n$-blocks of $g$ in level $0$. We can then replace $f$ by $\smat{1 & 0 \\ 0 & 1/p}(f-g)$ and the support of this new $f$ is included in the initial $\max(s_{j+1},n-j)$-block in level $j$ for $0 \leq j \leq n-2$ and in the initial $\max(s_0-n,1)$-block in level $n-1$ if $j=n-1$. By iterating the procedure of this paragraph, we can reduce the width of the support of $f$ until $s_j=n-j$ for $0 \leq j \leq n-1$.

The modified $f$ coming from the previous paragraph satisfies $\tau_n(f) - f \in S_n(h,\sigma)$ and its support is included in the initial $(n-j)$-block in level $j$ for $0 \leq j \leq n-1$. Corollary \ref{furtherwid} then shows that there exists $g \in S_n(h,\sigma)$ which is a linear combination of $\smat{1 & \Qp \\ 0 & 1}$-translates of $(-\lambda^{-1})^n \left[\smat{1 & 0 \\ 0 & p^n}\right] + w_{h_n,n}$ and of the $w_{\ell,n}$ for $0 \leq \ell \leq h_n-1$ such that the $n$-block of $g$ in level $0$ is the same as the $n$-block of $f$ in level $0$. We can then replace $f$ by $\smat{1 & 0 \\ 0 & 1/p}(f-g)$ and the support of this new $f$ is included in the initial $(n-j-1)$-block in level $j$ for $0 \leq j \leq n-1$. 

The modified $f$ coming from the previous paragraph satisfies $\tau_{n-1}(f) - f \in S_n(h,\sigma)$ and its support is included in the initial $(n-j-1)$-block in level $j$ for $0 \leq j \leq n-1$ and the $k$-block $x_k$ of $f$ in level $n-k-1$ is in $V_{h_k+1,k}$ by applying lemmas \ref{evenodd}, \ref{freeblock} and \ref{del}. By lemma \ref{freeblock}, we can subtract elements of $V_{h_k,k}$ from $x_k$ without changing the class of $f$ in $\Pi_n(h,\sigma)$ so we can assume that each $x_k$ is a (possibly $0$) multiple of $v_{h_k,k}$. If $0 \leq m \leq p-1$, let $U_m$ be the operator defined by $U_m(f) = \sum_{i=0}^{p-1} \binom{i}{m} \tau_n^i(f)$ as in the proof of lemma \ref{freeblock}. At level $n-1-k$ it has the effect of turning $v_{h_k,k}$ into $v_{h_{k+1}+ m - i_{n-k-1},k+1}$ since $h_{k+1} = ph_k +i_{n-k-1}$ and $\binom{j}{\ell} \binom{i}{m} = \binom{pj+i}{p\ell+m}$. If we choose $m$ such that $m - i_{n-k-1} \leq 0$ and $m - i_{n-k-1}=0$ for at least one value of $k$, then $U_m(f)$ is made up of $(k+1)$-blocks in level $n-k-1$ and we can get rid of all those for which $m - i_{n-k-1} \leq -1$. This allows us to lower the number of nonzero blocks of $f$ unless $m = i_{n-k-1}$ for all the corresponding nonzero blocks. In this case we lower $f$ by one level and if there is a block in level $0$ we send it to level $n$ before lowering $f$ by subtracting an appropriate multiple of $(-\lambda^{-1})^n \left[\smat{1 & 0 \\ 0 & p^n}\right] + w_{h_n,n}$. By iterating this procedure (replacing $f$ by $U_m(f)$ and lowering a possibly modified $f$), we can reduce the number of nonzero blocks of $f$ until our procedure starts cycling. 

If this is the case then there exists some $d$ dividing $n$ such that at some point $f$ has nonzero blocks exactly in levels $n-1-\ell d$ for $0 \leq \ell \leq (n/d)-1$ and the map $r \mapsto i_r$ is then also periodic of period $d$. If $d=n$, then we are done. If $d<n$, then I claim that $h_d$ is not divisible by $p-1$. Indeed, we have $h(p-1) = h_d(p^n-1)/(p^d-1)$ since $r \mapsto i_r$ is periodic of period $d$, and if $p-1$ divides $h_d$ then $h$ is not primitive. If $a \in \Zp^\times$ is such that $\overline{a}$ is a generator of $\Fp^\times$, then $\mu_a(v_{\ell,k}) - a^\ell v_{\ell,k} \in V_{\ell,k}$ by lemma \ref{mua}. This implies that $\sigma_2(a^{-1})\smat{1 & 0 \\ 0 & a} f - f$ has at least one fewer block (the top one) and is nonzero (the block of level $n-1-d$ is not in $S_n(h,\sigma)$), so that we can iterate again our procedure of the previous paragraph (replacing $f$ by $U_m(f)$ and lowering a possibly modified $f$) until $d=n$ so that $f$ becomes equivalent to an element supported on only one point.
\end{proof}

\begin{rema}\label{bze}
We have $\Pi_n(h,\sigma) \otimes (\chi \circ \det) \simeq \Pi_n(h,\sigma_1 \chi \otimes \sigma_2 \chi)$ by lemma \ref{intert}.
\end{rema}

\section{Galois representations and $(\phi,\Gamma)$-modules}
\label{secpgm}

In this chapter, we construct the $(\phi,\Gamma)$-modules associated to the absolutely irreducible $E$-linear representations of $\gal$ and then apply Colmez' functor to them in order to get a smooth irreducible representation of $\B$. 

\Subsection{Construction of $(\phi,\Gamma)$-modules}
\label{fontaine}

Let $\Cp$ be the completion of $\Qpbar$ and let $\et^+ = \projlim_{x \mapsto x^p} \OO_{\Cp}$ be the ring defined by Fontaine (see for example \S 1.2 of \cite{F3}). Recall that if $x,y \in \et^+$, then 
\[ (xy)^{(i)} = x^{(i)} y^{(i)} \quad\text{and}\quad (x+y)^{(i)}=\lim_{j \to \infty} (x^{(i+j)}+y^{(i+j)})^{p^j} \] 
and that $\et^+$ is endowed with the valuation $\ve$ defined by $\ve(y) = \val_p(y^{(0)})$. If we choose once and for all a compatible system $\{\zeta_{p^n}\}_{n \geq 0}$ of $p^n$-th roots of $1$ then $\eps=(1,\zeta_p,\zeta_{p^2},\hdots) \in \et^+$ and we set $X=\eps-1$ and $\et=\et^+[1/X]$ so that by \S 4.3 of \cite{W83}, $\et$ is an algebraically closed field of characteristic $p$, which contains $\Fp\dpar{X}^{\text{sep}}$ as a dense subfield. Given the construction of $\et$ from $\Cp$, we see that it is endowed with a continuous action of $\gal$. We have for instance $g(X)=(1+X)^{\chi_{\text{cycl}}(g)}-1$ if $g \in \gal$ so that $\hal = \ker \chi_{\text{cycl}}$ acts trivially on $\Fp\dpar{X}$ and we get a map $\hal \to \operatorname{Gal}(\Fp\dpar{X}^{\text{sep}} / \Fp\dpar{X})$ which is an isomorphism (this follows from the theory of the ``field of norms'' of \cite{FW79}, see for example theorem 3.1.6 of \cite{F90}). We also get an action of $\Gamma = \gal/\hal$ on $\Fp\dpar{X}$.

If $W$ is an $\Fp$-linear representation of $\gal$ then the $\Fp\dpar{X}$-vector space $\dfont(W) = (\Fp\dpar{X}^{\text{sep}} \otimes_{\Fp} W)^{\hal}$ inherits the frobenius $\phi$ of $\Fp\dpar{X}^{\text{sep}}$ and the residual action of $\Gamma$.

\begin{defi}
A $(\phi,\Gamma)$-\emph{module} over $\Fp\dpar{X}$ is a finite dimensional $\Fp\dpar{X}$-vector space endowed with a semilinear frobenius $\phi$ such that $\Mat(\phi) \in \GL_d(\Fp\dpar{X})$ and a continuous and semilinear action of $\Gamma$ commuting with $\phi$.
\end{defi}

We see that $\dfont(W)$ is then a $(\phi,\Gamma)$-module over $\Fp\dpar{X}$. If $E$ is a finite extension of $\Fp$, we endow it with the trivial $\phi$ and the trivial action of $\Gamma$ so that we may talk about $(\phi,\Gamma)$-modules over $E\dpar{X} = E \otimes_{\Fp} \Fp\dpar{X}$ and we then have the following result which is proved in \S 1.2 of \cite{F90} and whose proof we recall for the convenience of the reader.

\begin{theo}\label{fontequiv}
The functor $W \mapsto \dfont(W)$ gives an equivalence of categories between the category of $E$-representations of $\gal$ and the category of $(\phi,\Gamma)$-modules over $E\dpar{X}$.
\end{theo}

\begin{proof}[Sketch of proof]
Given the isomorphism $\hal \simeq \operatorname{Gal}(\Fp\dpar{X}^{\text{sep}} / \Fp\dpar{X})$, Hilbert's theorem 90 tells us that $\mathrm{H}^1_{\text{discrete}}(\hal, \GL_d(\Fp\dpar{X}^{\text{sep}})) = \{ 1 \}$ if $d \geq 1$ so that if $W$ is an $\Fp$-linear representation of $\hal$ then
\[ \Fp\dpar{X}^{\text{sep}} \otimes_{\Fp} W \simeq  (\Fp\dpar{X}^{\text{sep}})^{\dim(W)} \]
as representations of $\hal$ so that the $\Fp\dpar{X}$-vector space $\dfont(W) = (\Fp\dpar{X}^{\text{sep}} \otimes_{\Fp} W)^{\hal}$ is of dimension $\dim(W)$ and 
$W= (\Fp\dpar{X}^{\text{sep}} \otimes_{\Fp\dpar{X}} \dfont(W))^{\phi=1}$. 

If $\dfont$ is a $(\phi,\Gamma)$-module over $\Fp\dpar{X}$, then let $W(\dfont)=(\Fp\dpar{X}^{\text{sep}} \otimes_{\Fp\dpar{X}} \dfont)^{\phi=1}$. If we choose a basis of $\dfont$ and if $\Mat(\phi)=(p_{ij})_{1 \leq i,j \leq \dim(\dfont)}$ in that basis, then the algebra
\[ A = \Fp\dpar{X}[X_1,\hdots,X_{\dim(\dfont)}] / (X_j^p - \sum_i p_{ij} X_i)_{1 \leq j \leq \dim(\dfont)} \] 
is an \'etale $\Fp\dpar{X}$-algebra of rank $p^{\dim(\dfont)}$ and $W(\dfont)=\operatorname{Hom}_{\Fp\dpar{X}-\text{algebra}}(A,\Fp\dpar{X}^{\text{sep}})$ so that $W(\dfont)$ is an $\Fp$-vector space of dimension $\dim(\dfont)$. 

It is then easy to check that the functors $W \mapsto \dfont(W)$ and $\dfont \mapsto W(\dfont)$ are inverse of each other. Finally, if $E\neq \Fp$ then one can consider an $E$-representation as an $\Fp$-representation with an $E$-linear structure and likewise for $(\phi,\Gamma)$-modules, so that the equivalence carries over.
\end{proof}

We now compute the $(\phi,\Gamma)$-modules associated to certain Galois representations. If $n$ is an integer $\geq 1$, choose $\pi_n \in \Qpbar$ such that $\pi_n^{p^n-1} = -p$. The fundamental character of level $n$ defined in \S 1.7 of \cite{S72}, $\omega_n : \inert \to \Fpbar^\times$ is given by $\omega_n(g) = \overline{g(\pi_n)/\pi_n} \in \Fpbar^\times$ for $g\in \inert$. This definition does not depend on the choice of $\pi_n$ and shows that $\omega_n$ extends to a character $\galpn \to \Fpn^\times$. With this definition, $\omega_n$ is actually the reduction mod $p$ of the Lubin-Tate character associated to the uniformizer $p$ of the field $\Qpn$. 

In order to describe the $(\phi,\Gamma)$-modules associated to irreducible mod $p$ representations, we need to give a ``characteristic $p$'' construction of $\omega_n$. Let $\omega=\omega_1$ be the mod $p$ cyclotomic character and let $Y \in\Fp\dpar{X}^{\text{sep}}$ be an element such that $Y^{(p^n-1)/(p-1)} = X$.  If $g \in \gal$, then $f_g(X) = \omega(g)X/g(X)$ depends only on the image of $g$ in $\Gamma$. Since $f_g(X) \in 1 + X \Fp\dcroc{X}$, the formula $f_g^s(X)$ makes sense if $s \in \Zp$.

\begin{lemm}\label{yon}
If $g \in \galpn$ then $g(Y) = Y \omega_n^p(g) f_g^{-\frac{p-1}{p^n-1}}(X)$. 
\end{lemm}

\begin{proof}
Recall that $X \in \et^+ = \projlim \OO_{\Cp}$ is equal to $\eps-1$ where $\eps=(\zeta_{p^j})_{j \geq 0}$ and where $\{\zeta_{p^j}\}_{j \geq 0}$ is a compatible sequence. If $j \geq 1$, pick $\pi_{n,j} \in \OO_{\Cp}$ such that
\[ \pi_{n,j}^{\frac{p^n-1}{p-1}} = \zeta_{p^j}-1. \]
If $g \in \galpn$, then $g(\zeta_{p^j}-1)=[\omega(g)] (\zeta_{p^j}-1)f_g^{-1}(\zeta_{p^j}-1)$ where we also write $f_g(X)$ for $[\omega(g)] X/((1+X)^{\chi_{\text{cycl}}(g)}-1) \in 1+X\Zp\dcroc{X}$ and so there exists $\omega_{n,j}(g) \in \Fpn^\times$ such that
\[ \frac{g(\pi_{n,j})}{\pi_{n,j}} = [\omega_{n,j}(g)] f_g^{-\frac{p-1}{p^n-1}}(\zeta_{p^j}-1), \]
where $[\cdot]$ is the Teichm\"uller lift from $\Fpn^\times$ to $\Qpn^\times$. The map $g \mapsto \omega_{n,j}(g)$ is a character of $\galpn$ which does not depend on the choice of $\pi_{n,j}$. In addition, we have
\[ \begin{cases} (\zeta_{p^{j+1}}-1)^p=(\zeta_{p^j}-1) \cdot (1+\mathrm{O}(p^{1/p})) & \text{if $j \geq 1$,} \\ 
(\zeta_p-1)^{p-1} = -p \cdot (1+\mathrm{O}(p^{1/p}))
\end{cases} \]
so that $\omega_{n,j+1}^p=\omega_{n,j}$ if $j \geq 1$ and $\omega_{n,1} = \omega_n$. This also tells us that we may choose the $\pi_{n,j}$ so that $\pi_{n,j+1}^p / \pi_{n,j} \in 1 + p^{1/p} \OO_{\Cp}$. If we write $Y=(y^{(i)}) \in \projlim \OO_{\Cp}$, then we have $y^{(i)} = \lim_{j \to +\infty} \pi_{n,i+j}^{p^j}$ since the $\pi_{n,j}$ are compatible in the sense that $\pi_{n,j+1}^p / \pi_{n,j} \in 1 + p^{1/p} \OO_{\Cp}$ so that if $g \in \galpn$, then
\[ \frac{g(y^{(i)})}{y^{(i)}} = [\omega_{n,i}(g)] \cdot \lim_{j \to + \infty} (f_g^{-\frac{p-1}{p^n-1}}(\zeta_{p^{i+j}}-1))^{p^j}, \]
and therefore we have $g(Y)=Y \omega_n^p(g) f_g^{-\frac{p-1}{p^n-1}}(X)$ in $\et$.
\end{proof}

If $1 \leq h \leq p^n-2$ is primitive, the characters $\omega_n^h, \omega_n^{ph}, \hdots, \omega_n^{p^{n-1} h}$ of $\inert$ are pairwise distinct. Let $\mu_\lambda$ be the unramified character sending the arithmetic frobenius to $\lambda^{-1}$ (so that later when we normalize class field theory to send the geometric frobenius to $p$ then $\mu_\lambda(p) = \lambda$).

\begin{lemm}\label{lemmomn}
Every absolutely irreducible $n$-dimensional $E$-linear representation of $\gal$ is isomorphic (after possibly enlarging $E$) to $(\ind_{\galpn}^{\gal} \omega_n^h) \otimes \mu_\lambda$ for some primitive $1 \leq h \leq p^n-2$ and some $\lambda \in E^\times$.
\end{lemm}

\begin{proof}
If $W$ is such a representation then by \S 1.6 of \cite{S72}, we may extend $E$ so that $W|_{\inert}$ splits as a direct sum of $n$ tame characters and since $W$ is irreducible, these characters are  transitively permuted by frobenius so that they are of level $n$ and there exists a primitive $h$ such that $W = \oplus_{i=0}^{n-1} W_i$ where $\inert$ acts on $W_i$ by $\omega_n^{p^i h}$. Since $\omega_n$ extends to $\galpn$ each $W_i$ is stable under $\galpn$ which then acts on it by $\omega_n^{p^i h} \chi_i$ where $\chi_i$ is an unramified character of $\galpn$. The lemma then follows from Frobenius reciprocity. 
\end{proof}

If $\lambda \in \Fpbar^\times$ is such that $\lambda^n \in \Fp^\times$, let $W_\lambda = \{ \alpha \in \Fpbar$ such that $\alpha^{p^n} = \lambda^{-n} \alpha\}$ so that $W_\lambda$ is a $\Fpn$-vector space of dimension $1$ and hence a $\Fp$-vector space of dimension $n$. By composing the map $\Gal(\Qp^\text{nr}(\pi_n) / \Qp) \xrightarrow{\sim} \Fpn^\times \rtimes \hat{\ZZ}$ with the map $\Fpn^\times \rtimes \hat{\ZZ} \to \operatorname{End}_{\Fp}(W_\lambda)$ given by $(x,0) \mapsto m_x^h$ (where $m_x$ is the multiplication by $x$ map) and by $(1,1) \mapsto (\alpha \mapsto \alpha^p)$ we get an $n$-dimensional $\Fp$-linear representation of $\gal$ which is isomorphic to $(\ind_{\galpn}^{\gal} \omega_n^h) \otimes \mu_\lambda$ after extending scalars and whose determinant is $\omega^h\mu_{-1}^{n-1} \mu_{\lambda}^n$ so that if $\lambda^n=(-1)^{n-1}$ then the determinant is $\omega^h$ and we call $\ind(\omega_n^h)$ the representation thus constructed; it is then uniquely determined by the two conditions $\det \ind(\omega_n^h) = \omega^h$ and $\ind(\omega_n^h) |_{\inert} =  \oplus_{i=0}^{n-1} \omega_n^{p^i h}$ since $(\ind_{\galpn}^{\gal} \omega_n^h) \otimes \mu_{\lambda_1}= (\ind_{\galpn}^{\gal} \omega_n^h) \otimes \mu_{\lambda_2}$ if and only if we have $\lambda_1^n=\lambda_2^n$. 

\begin{coro}\label{twistomn}
Every absolutely irreducible $n$-dimensional $E$-linear representation of $\gal$ is isomorphic to $\ind(\omega_n^h) \otimes \mu_\lambda$ for some primitive $1 \leq h \leq p^n-2$ and some $\lambda \in \Fpbar^\times$ such that $\lambda^n \in E^\times$.
\end{coro}

\begin{theo}\label{pgomegan}
The $(\phi,\Gamma)$-module $\dfont(\ind(\omega_n^h))$ is defined over $\Fp\dpar{X}$ and admits a basis $e_0, \hdots, e_{n-1}$ in which $\gamma(e_j) =  f_\gamma(X) ^{hp^j(p-1)/(p^n-1)} e_j$ if $\gamma \in \Gamma$ and $\phi(e_j) = e_{j+1}$ for $0 \leq j \leq n-2$ and $\phi(e_{n-1}) = (-1)^{n-1} X^{-h(p-1)} e_0$.
\end{theo}

\begin{proof} 
Let $W$ be the $\Fp$-representation of $\gal$ associated to the $(\phi,\Gamma)$-module described in the theorem. If $f = X^h e_0 \wedge \hdots \wedge e_{n-1}$, then $\phi(f)=f$ and $\gamma(f) = \omega(\gamma)^h f$ so that the determinant of $W$ is indeed $\omega^h$ and therefore we only need to show that the restriction of $\Fpn \otimes_{\Fp} W$ to $\inert$ is $\omega_n^h \oplus \omega_n^{ph} \oplus \cdots \oplus \omega_n^{p^{n-1} h}$. To clarify things, let us write $\Fpn^{\natural}$ for $\Fpn$ when it occurs as a coefficient field, so that $\phi$ is trivial on $\Fpn^{\natural}$.

If we write $\Fpn^{\natural} \otimes_{\Fp} \Fp\dpar{X}^{\text{sep}}$ as $\prod_{k=0}^{n-1} \Fp\dpar{X}^{\text{sep}}$ via the map $x \otimes y \mapsto (\sigma^k(x)y)$ where $\sigma$ is the absolute frobenius on $\Fpn^{\natural}$, then given $(x_0,\hdots,x_{n-1}) \in \prod_{k=0}^{n-1} \Fp\dpar{X}^{\text{sep}}$, we have
\begin{align*}
\phi((x_0,\hdots,x_{n-1})) & =(\phi(x_{n-1}),\phi(x_0),\hdots,\phi(x_{n-2})) \\ 
g((x_0,\hdots,x_{n-1}))& =(g(x_0),\hdots,g(x_{n-1})),
\end{align*} 
if $g \in \galpn$ (but not if $g \in \gal$). Choose some $\alpha \in \Fp\dpar{X}^{\text{sep}}$ such that $\alpha^{p^n-1} = (-1)^{n-1}$ and define
\begin{align*}
v_0 & = (\alpha Y^h,0,\hdots,0) \cdot e_0 + (0, \alpha^p Y^{ph}, \hdots,0) \cdot e_1 + \cdots (0,\hdots,0,\alpha^{p^{n-1}} Y^{p^{n-1}h}) \cdot e_{n-1} \\
v_1 & = (0,\alpha Y^h,\hdots,0) \cdot e_0 + (0,0,  \alpha^p Y^{ph}, \hdots,0) \cdot e_1 + \cdots (\alpha^{p^{n-1}} Y^{p^{n-1}h},0,\hdots,0) \cdot e_{n-1} \\
 & \vdots \\
v_{n-1} &= (0,\hdots,0,\alpha Y^h) \cdot e_0 + (\alpha^p Y^{ph},0, \hdots,0) \cdot e_1 + \cdots (0,\hdots,0,\alpha^{p^{n-1}} Y^{p^{n-1}h},0) \cdot e_{n-1}.
\end{align*}
The vectors $v_0,\hdots,v_{n-1}$ give a basis of $\Fpn^{\natural} \otimes_{\Fp} (\Fp\dpar{X}^{\text{sep}} \otimes_{\Fp\dpar{X}} \dfont(W))$ and the formulas for the action of $\phi$ imply that $\phi(v_j)=v_j$ so that $v_j \in \Fpn^{\natural} \otimes_{\Fp} W$. The formulas for the action of $\Gamma$ and lemma \ref{yon} imply that $g(v_j) = \omega_n^{hp^{1-j}} v_j$ if $g \in \inert$ which finishes the proof.
\end{proof}

\Subsection{From Galois to Borel}\label{colmez}
If $\alpha(X) \in E\dpar{X}$ then we can write
\[ \alpha(X) = \sum_{j=0}^{p-1} (1+X)^j \alpha_j(X^p) \] 
in a unique way, and we define a map $\psi : E\dpar{X} \to E\dpar{X}$ by the formula $\psi(\alpha)(X) = \alpha_0(X)$. A direct computation shows that if $0 \leq r \leq p-1$ then $\psi(X^{pm+r}) = (-1)^r X^m$. If $\dfont$ is a $(\phi,\Gamma)$-module over $E\dpar{X}$ and if $y \in \dfont$ then likewise we can write $y= \sum_{j=0}^{p-1} (1+X)^j \phi(y_j)$ and we set $\psi(y) = y_0$. The operator $\psi$ thus defined commutes with the action of $\Gamma$ and satisfies $\psi(\alpha(X) \phi(y)) = \psi(\alpha)(X) y$ and $\psi(\alpha(X^p)y) = \alpha(X) \psi(y)$. 

If $W = \ind(\omega_n^h) \otimes \chi$ with $\chi = \omega^s \mu_{\lambda}$ where from now on $\lambda \in E^\times$, then theorem \ref{pgomegan} above implies that the $(\phi,\Gamma)$-module $\dfont(W)$ is defined on $E\dpar{X}$ and admits a basis $e_0, \hdots, e_{n-1}$ in which $\gamma(e_j) =  \omega^s(\gamma) f_\gamma(X) ^{hp^j(p-1)/(p^n-1)} e_j$ if $\gamma \in \Gamma$ and $\phi(e_j) = \lambda e_{j+1}$ for $0 \leq j \leq n-2$ and $\phi(e_{n-1}) = (-1)^{n-1} \lambda X^{-h(p-1)} e_0$. Since $\omega_n^{(p^n-1)/(p-1)} = \omega$, we can always modify $h$ (and $\chi$ accordingly) in order to have $1 \leq h \leq (p^n-1)/(p-1) -1$ so that $h(p-1) \leq p^n-2$. Recall that $i_{n-1} \hdots i_1 i_0$ is the expansion of $h(p-1)$ in base $p$ and that $h_k = i_{n-k}  + p i_{n-k+1} + \cdots + p^{k-1} i_{n-1}$ so that $h_0=0$ and $h_n=h(p-1)$. 

\begin{lemm}\label{psiomeg}
If $f_j  = X^{h_j} e_j $ and $\alpha(X) \in E \dpar{X}$, then we have
\[ \psi(\alpha(X) f_j) = \begin{cases}
\lambda^{-1} \psi(\alpha(X) X^{i_{n-j}}) f_{j-1} & \text{if $j \geq 1$,} \\
\lambda^{-1} (-1)^{n-1} \psi(\alpha(X) X^{i_0}) f_{n-1} & \text{if $j = 0$.}
\end{cases} \]
\end{lemm}

\begin{proof}
If $j \geq 1$, then we can write $\alpha(X) f_j = \lambda^{-1} \alpha(X) X^{h_j} \phi(e_{j-1})$ and since $h_j = ph_{j-1} + i_{n-j}$, we have
\[ \psi(\alpha(X) f_j ) = \lambda^{-1} X^{h_{j-1}} \psi(\alpha(X) X^{i_{n-j}}) e_{j-1} = \lambda^{-1} \psi(\alpha(X) X^{i_{n-j}}) f_{j-1}. \]
If $j=0$, then $\alpha(X) f_0 = \alpha(X)e_0 = \alpha(X) (-1)^{n-1} \lambda^{-1} X^{h(p-1)} \phi(e_{n-1})$ so that
\[ \psi(\alpha(X) f_0 ) = \lambda^{-1} (-1)^{n-1} X^{h_{n-1}} \psi(\alpha(X) X^{i_{0}}) e_{n-1} = \lambda^{-1} (-1)^{n-1} \psi(\alpha(X) X^{i_0}) f_{n-1} \]
which finishes the proof.
\end{proof}

\begin{coro}\label{ddiese}
The $E\dcroc{X}$-module $\ddiese(W) = \oplus_{j=0}^{n-1} E\dcroc{X} \cdot f_j$ is stable under $\psi$ and the map $\psi : \ddiese(W) \to \ddiese(W)$ is surjective.
\end{coro}

\begin{proof}
Lemma \ref{psiomeg} implies that $\ddiese(W)$ is stable under $\psi$. Furthermore, the formula $\psi(X^{pm+r}) = (-1)^r X^m$ for $0 \leq r \leq p-1$ implies that the map $\alpha_j(X) \mapsto \psi(\alpha_j(X) X^{i_{n-j}})$ is surjective for $j \geq 1$, as well as the map $\alpha_0(X) \mapsto \psi(\alpha_0(X) X^{i_0})$, which implies that $\psi : \ddiese(W) \to \ddiese(W)$ is surjective.
\end{proof}

A quick computation shows that if $y \in \ddiese(W)$ then $\psi^n(X^{-1} y) \in \ddiese(W)$ so that our $\ddiese(W)$ coincides with the lattice defined by Colmez in proposition II.4.2 of \cite{CPG} by item (iv) of that proposition. We now define Colmez' functor (see \S III of \cite{CPG}):
\[ \projlim_{\psi} \ddiese(W) = \text{$\{ y = (y_0,y_1,\hdots)$ with $y_i \in \ddiese(W)$ such that $\psi(y_{i+1}) = y_i$ for all $i\geq 0 \}$,} \] 
and we endow this space with an action of $\B$ (using the same normalization as in \cite{LB5} which differs by a twist from the normalization of \cite{CPG})
\begin{align*}
\left( \begin{pmatrix} x & 0 \\ 0 & x \end{pmatrix} \cdot y \right)_i
& = (\omega^{h-1} \chi^2)^{-1}(x) y_i;  \\
\left( \begin{pmatrix} 1 &  0 \\ 0 & p^j  
\end{pmatrix} \cdot y \right)_i & = y_{i-j} = \psi^j(y_i); \\
\left( \begin{pmatrix} 1 &  0 \\ 0 & a  
\end{pmatrix} \cdot y \right)_i & = \gamma_{a^{-1}}(y_i), 
\text{ where $\gamma_{a^{-1}} \in \Gamma$ is such that $\chi_{\text{cycl}}(\gamma_{a^{-1}}) = a^{-1} \in \Zp^\times$;}\\
\left( \begin{pmatrix} 1 &  z \\ 0 & 1  
\end{pmatrix} \cdot y \right)_i & = 
\psi^j((1+X)^{p^{i+j} z} y_{i+j}),\text{ for $i+j \geq -{\rm val}(z)$.}
\end{align*}
We then define $\Omega(W) = (\projlim_{\psi} \ddiese(W))^*$ so that $\Omega(W)$ is a smooth representation (see \S\ref{duality} for a proof of this) of $\B$ whose central character is $\omega^{h-1} \chi^2$. Denote by $\theta_0$ the linear form on $\ddiese(W)$ given by
\[ \theta_0 : \alpha_0(X)f_0 + \cdots + \alpha_{n-1}(X)f_{n-1} \mapsto \alpha_0(0). \] 
If $y = (y_0,y_1,\hdots)$,  then we define $\theta \in \Omega(W)$ to be the linear form $\theta : y \mapsto \theta_0(y_0)$. 

\begin{lemm}\label{acbormu}
If $\smat{ a & b \\ 0 & d} \in \K \Z$, then $\smat{ a & b \\ 0 & d} \cdot \theta = \omega^{h-1}(a) \chi(ad) \theta$.
\end{lemm}

\begin{proof}
We have
\begin{align*} 
\left(\pmat{ a & b \\ 0 & d} \cdot \theta\right)(y) 
& = \theta \left( \pmat{ a^{-1} & -b a^{-1} d^{-1} \\ 0 & d^{-1}} \cdot y \right) \\
& = \theta \left( \pmat{ a^{-1} & 0 \\ 0 & a^{-1}} \pmat{ 1 & -b d^{-1} \\ 0 & ad^{-1}} \cdot y \right) \\
& = (\omega^{h-1} \chi^2)(a) \omega^s(a^{-1}d) \theta(y) \\
&=  \omega^{h-1}(a) \chi(ad) \theta(y),
\end{align*}
since $\mu_\lambda(a) = \mu_\lambda(d)$ because $\smat{ a & b \\ 0 & d} \in \K \Z$ so that $\chi(a) = \chi(d) \omega^s(ad^{-1})$.
\end{proof}

For $0 \leq k \leq n$, recall that $h_k = i_{n-k}  + p i_{n-k+1} + \cdots + p^{k-1} i_{n-1}$ so that $h_n=h(p-1)$.

\begin{theo}\label{nultheta}
The linear form $\theta$ is killed by 
\[ (-1)^{n-1}  \lambda^n \cdot \Id - \sum_{j=0}^{p^n-1} \binom{j}{h(p-1)} \pmat{p^n & -j \\ 0 & 1}. \]
\end{theo}

\begin{proof}
Using the definition of the action of $\B$ on $\projlim_{\psi} \ddiese(W)$, we get
\begin{multline*} 
\left((-1)^{n-1}  \lambda^n \cdot \theta - \sum_{j=0}^{p^n-1} \binom{j}{h(p-1)} \pmat{p^n & -j \\ 0 & 1} \theta \right)(y)  \\
= (-1)^{n-1}  \lambda^n \cdot \theta_0(y_0) -  \lambda^{2n}  \cdot \theta_0 \circ \psi^n \left( \sum_{j=0}^{p^n-1} \binom{j}{h(p-1)} (1+X)^j y_0 \right), 
\end{multline*}
and this is equal to $0$ for obvious reasons if $y_0=\alpha_i(X)f_i$ with $i \neq 0$ so that we now assume that $y_0 = \alpha_0(X)f_0$. Lemma \ref{prbin} implies that
\[ \sum_{j=0}^{p^n-1} \binom{j}{h(p-1)} (1+X)^j \in  X^{p^n-h_n-1} + X^{p^n-h_n} E\dcroc{X}, \]
and the fact that $p^\ell-h_\ell+i_{n-\ell} =p(p^{\ell-1}-h_{\ell-1})$ for $1 \leq \ell \leq n$ together with the formulas of lemma \ref{psiomeg} and the fact that $\psi(X^{pm+r}) = (-1)^r X^m$ then imply that
\[ \psi^n \left( \sum_{j=0}^{p^n-1} \binom{j}{h(p-1)} (1+X)^j  \alpha_0(X)f_0 \right) \equiv (-1)^{n-1}  \lambda^{-n}  \alpha_0(X) f_0 \mod{X \ddiese(W)}, \]
which proves our claim.
\end{proof}

\Subsection{Profinite representations and smooth representations}
\label{duality}
In this paragraph, we prove that $\Omega(W)$ is a smooth irreducible representation of $\B$ if $\dim(W) \geq 2$. In order to do so, we recall a few results concerning profinite representations and their dual. Let $G$ be a topological group and let $X$ be a profinite $E$-linear representation of $G$ where $E$ is as before a finite extension of $\Fp$. Let $X^*$ be the dual of $X$, that is the set of continuous linear forms on $X$.

\begin{lemm}\label{dualsm}
The representation $X^*$ is a smooth representation of $G$.
\end{lemm}

\begin{proof}
If $f \in X^*$, then the map $(g,x) \mapsto f(gx-x)$ is a continuous map $G \times X \to E$ and its kernel is therefore open in $G \times X$ so that there exists an open subgroup $K$ of $G$ and an open subspace $Y$ of $X$ such that $f(ky-y)=0$ whenever $k \in K$ and $y \in Y$. Since $X$ is compact, $Y$ is of finite codimension in $X$ and we can write $X= Y \oplus \oplus_{i=1}^s E x_i$. For each $i$ there is an open subgroup $K_i$ of $G$ such that $f(k x_i - x_i)=0$ if $k \in K_i$ and this implies that if $H = K \cap \cap_{i=1}^s K_i$ then $f(hx-x)=0$ for any $x \in X$ so that $f \in (X^*)^H$ with $H$ an open subgroup of $G$.
\end{proof}

\begin{lemm}\label{dualirred}
If $X$ is topologically irreducible, then $X^*$ is irreducible.
\end{lemm}

\begin{proof}
If $X=\projlim_{i \in I} X_i$ where each $X_i$ is a finite dimensional $E$-vector space, then a linear form on $X$ is continuous if and only if it factors through some $X_i$ and hence $X^*= \varinjlim_{i \in I} X_i^*$ so that $(X^*)^* = (\varinjlim_{i \in I} X_i^*)^*=  \projlim_{i \in I} X_i=X$. If $\Lambda$ is a $G$-invariant subspace of $X^*$ then $\ker(\Lambda) = \cap_{f \in \Lambda} \ker(f)$ is a $G$-invariant closed subspace of $X$ which is therefore either equal to $X$ or to $\{0\}$. If it is equal to $X$ then obviously $\Lambda=\{0\}$ and if it is equal to $\{0\}$, then the fact that $(X^*)^*=X$ implies that no nonzero linear form on $X^*$ is zero on $\Lambda$ so that $\Lambda=X^*$.
\end{proof}

The representation $\projlim_\psi \ddiese(W)$ is a profinite representation of $\B$ since $\ddiese(W) \simeq E\dcroc{X}^{\dim(W)}$ and we have the following result (see also proposition 1.2.3 of \cite{LB5}). 

\begin{prop}\label{omsmirr}
The representation $\Omega(W) = (\projlim_\psi \ddiese(W))^*$ is a smooth irreducible representation of $\B$ if $\dim(W) \geq 2$.
\end{prop}

\begin{proof}
Lemma \ref{dualirred} shows that it is enough to prove that $\projlim_\psi \ddiese(W)$ is a topologically irreducible representation of $\B$, and lemma III.3.6 of \cite{CPG} asserts that any closed $\B$-invariant subspace of $\projlim_\psi \ddiese(W)$ is of the form $\projlim_\psi M$ where $M$ is a sub-$E\dcroc{X}$-module of $\ddiese(W)$ stable under $\psi$ and $\Gamma$ and such that $\psi : M \to M$ is surjective. Since $\dfont(W)$ is irreducible, $M$ is a lattice by proposition II.3.5 of \cite{CPG} applied to $E\dpar{X} \otimes_{E\dcroc{X}} M$ and item (iv) of proposition II.4.2 of \cite{CPG} implies that such an $M$ contains $X \cdot \ddiese(W)$ and the formulas of lemma \ref{psiomeg} imply that $\psi(X f_j) \in E^\times \cdot f_{j-1}$ if $i_{n-j} \neq p-1$. Since $h(p-1) \neq p^n-1$, at least one of the $i_{n-j}$ is $\neq p-1$ so that $M$ contains one $f_j$ and hence all of them by repeatedly applying $\psi$. 
\end{proof}

\section{Breuil's correspondence for mod $p$ representations}
\label{breuil}

In this chapter, we show that the representations constructed in chapter \ref{smr} are the same as the ones arising from Colmez' functor applied to $n$-dimensional absolutely irreducible representations of $\gal$. We also show that if $n=2$, then these representations are the restriction to $\B$ of the supersingular representations of $\G$ predicted by Breuil.

\Subsection{The isomorphism in dimension $n$}\label{dimde}
By corollary \ref{twistomn}, every absolutely irreducible $n$-dimensional $E$-linear representation $W$ of $\gal$ is isomorphic (after possibly enlarging $E$) to $\ind(\omega_n^h) \otimes \chi$ with $1 \leq h \leq p^n-2$ primitive and $\chi : \gal \to E^\times$ a character. Furthermore, $\omega_n^{(p^n-1)/(p-1)} = \omega$ so we can change $h$ and $\chi$ in order to have $1 \leq h \leq (p^n-1)/(p-1)-1$ which implies that at least one of the $n$ digits of $h$ in base $p$ is zero. The intertwining $\ind(\omega_n^h) \simeq \ind(\omega_n^{ph})$ implies that we can make a cyclic permutation of the digits of $h$ without changing $\ind(\omega_n^h)$ and if we arrange for the leading digit to be $0$, then $1 \leq h \leq p^{n-1} - 1$.

\begin{theo}\label{shlisom}
If $W = \ind(\omega_n^h) \otimes \chi$ with $n \geq 2$ and $1 \leq h \leq p^{n-1}-1$ primitive, then $\Omega(W) \simeq \Pi_n(h,\sigma)$ with $\sigma= \chi\omega^{h-1} \otimes \chi$.
\end{theo}

\begin{proof}
By lemma \ref{acbormu} and Frobenius reciprocity, $\Omega(W)$ is a quotient of $\ind_{\K\Z}^{\B} \sigma$ with $\sigma= \chi\omega^{h-1} \otimes \chi$, the map being given by $\sum_{\beta,\delta} \alpha(\beta,\delta) [ g_{\beta,\delta}] \mapsto \sum_{\beta,\delta} \alpha(\beta,\delta) g_{\beta,\delta} \cdot \theta$. This map is surjective (since it is nonzero and $\Omega(W)$ is irreducible by proposition \ref{omsmirr}) and bearing in mind that $\smat{p^n & 0 \\ 0 & p^n}$ acts by $\lambda^{2n}$, theorem \ref{nultheta} implies that its kernel contains $(-\lambda^{-1})^n \left[\smat{1 & 0 \\ 0 & p^n}\right] + w_{h(p-1),n}$ and hence $S_n(h,\sigma)$, so that we get a nontrivial map $\Pi_n(h,\sigma) \to \Omega(W)$. Since $\Pi_n(h,\sigma)$ is irreducible by theorem \ref{irredim2}, this map is an isomorphism.
\end{proof}

Note that we can define two $B$-equivariant operators $T_+$ and $T_-$ on $\ind_{\K\Z}^{\B} \sigma$ by
\[ T_+([g]) = \sum_{j=0}^{p-1} \left[g\smat{p & j \\ 0 & 1}\right]
\quad\text{and}\quad 
T_-([g]) = \left[g\smat{1 & 0 \\ 0 & p}\right], \]
so that the ``usual'' Hecke operator is $T=T_+ + T_-$. It is easy to see that theorem \ref{shlisom} applied with $h=1$ simply says that 
\[ \Omega(\ind(\omega_n) \otimes \chi) \simeq \frac{\ind_{\K\Z}^{\B} (1 \otimes 1)}{T_- + (-1)^n T_+^{n-1}} \otimes (\chi \circ \det). \]

\begin{center}
\begin{picture}(300,120)(0,10)

\put(0,10){\framebox(300,120){}}
\put(5,50){ht $\delta$}
\put(5,100){ht $\delta+1$}

\dottedline{2}(60,50)(120,50)
\dottedline{2}(60,100)(120,100)

\dottedline{2}(180,50)(240,50)
\dottedline{2}(180,100)(240,100)

\put(90,100){\circle*{5}}

\put(70,50){\line(2,5){20}}
\put(90,50){\line(0,5){50}}
\put(110,50){\line(-2,5){20}}

\put(190,50){\line(2,5){20}}
\put(210,50){\line(0,5){50}}
\put(230,50){\line(-2,5){20}}

\put(190,50){\circle*{5}}
\put(210,50){\circle*{5}}
\put(230,50){\circle*{5}}

\put(95,105){$x$}

\put(190,35){$x$}
\put(210,35){$x$}
\put(230,35){$x$}

\put(150,75){$\mapsto$}

\put(150,25){$T_+$}

\end{picture}
\end{center}

\begin{center}
\begin{picture}(300,120)(0,10)

\put(0,10){\framebox(300,120){}}
\put(5,50){ht $\delta$}
\put(5,100){ht $\delta+1$}

\dottedline{2}(60,50)(120,50)
\dottedline{2}(60,100)(120,100)

\dottedline{2}(180,50)(240,50)
\dottedline{2}(180,100)(240,100)

\put(210,100){\circle*{5}}

\put(70,50){\line(2,5){20}}
\put(90,50){\line(0,5){50}}
\put(110,50){\line(-2,5){20}}

\put(190,50){\line(2,5){20}}
\put(210,50){\line(0,5){50}}
\put(230,50){\line(-2,5){20}}

\put(70,50){\circle*{5}}
\put(90,50){\circle*{5}}
\put(110,50){\circle*{5}}

\put(150,75){$\mapsto$}

\put(215,105){$x+y+z$}

\put(70,35){$x$}
\put(90,35){$y$}
\put(110,35){$z$}

\put(150,25){$T_-$}

\end{picture}
\end{center}

\Subsection{Supersingular representations restricted to $\B_2(\Qp)$}
\label{ssgresb}
We now explain how to relate the representations $\Pi_2(h,\sigma)$ to the supersingular representations of \cite{BL,BL2,BR1}. Recall that if $r \geq 0$, then $\Sym^r E^2$ is the space of polynomials in $x$ and $y$ which are homogeneous of degree $r$ with coefficients in $E$, endowed with the action of $\KK$ factoring through $\operatorname{GL}_2(\Fp)$ given by $\smat{a & b \\ c & d} P(x,y) = P(ax+cy,bx+dy)$ and that we extend the action of $\KK$ to an action of $\KK\Z$ by $\smat{p & 0 \\ 0 & p} P(x,y) = P(x,y)$.  We now assume that $0 \leq r \leq p-1$.

\begin{lemm}\label{iwasawa}
The ``restriction to $\B$'' map
\[ \res_{\B} : \ind_{\KK\Z}^{\G} \Sym^r E^2 \to \ind_{\K\Z}^{\B} \Sym^r E^2 \] 
is an isomorphism. 
\end{lemm}

\begin{proof}
This follows from the Iwasawa decomposition $\G=\B \cdot \KK$.
\end{proof}

Let $T$ be the Hecke operator defined in \cite{BL,BL2}. Let $[g,v] \in \ind_{\KK\Z}^{\G} \Sym^r E^2$ be the element defined by $[g,v](h)= \Sym^r(hg)(v)$ if $hg \in\KK\Z$ and $[g,v](h)=0$ otherwise, so that $h[g,v]=[hg,v]$ and $[gk,v]=[g,\Sym^r(k)v]$ if $k \in\KK\Z$. 

\begin{lemm}\label{formule}
We have
\[ T ( [1,x^{r-i}y^i] ) =  
\begin{cases} \sum_{j =0}^{p-1}  \smat{p & j \\ 0 & 1} [1, (-j)^i x^r] & \text{if $i \leq r-1$;} \\
\smat{1 & 0 \\ 0 & p} [1,y^r] + \sum_{j =0}^{p-1}  \smat{p & j \\ 0 & 1} [1, (-j)^r x^r] & \text{if $i = r$.}
\end{cases} \]
\end{lemm}

\begin{proof}
See \S 2.2 of \cite{BR2}.
\end{proof}

The group $\K\Z$ acts on $x^r \in \Sym^r E^2$ by $\omega^r \otimes 1$ so that we get a nontrivial injective map  $\ind_{\K\Z}^{\B} \omega^r \otimes 1 \to \ind_{\K\Z}^{\B} \Sym^r E^2$.

\begin{prop}\label{isombor}
The map
\[ \frac{\ind_{\K\Z}^{\B} (\omega^r \otimes 1)}{T (\ind_{\K\Z}^{\B} \Sym^r E^2) \cap \ind_{\K\Z}^{\B} (\omega^r \otimes 1)}   \to \frac{\ind_{\K\Z}^{\B} \Sym^r E^2}{T(\ind_{\K\Z}^{\B} \Sym^r E^2)} \]
is an isomorphism.
\end{prop}

\begin{proof}
The map above is injective by construction, and the representation to the right is generated by the $\B$-translates of $[1,y^r]$ since the $\smat{1 & \Zp \\ 0 & 1}$-translates of $y^r$ generate $\Sym^r E^2$. Lemma \ref{formule} applied with $i=r$ shows that $[1,y^r] \in T(\ind_{\K\Z}^{\B} \Sym^r E^2)+\ind_{\K\Z}^{\B} (\omega^r \otimes 1)$ so that the map is surjective.
\end{proof}

\begin{lemm}\label{noyr}
If $r \geq 1$, then $T (\ind_{\K\Z}^{\B} \Sym^r E^2) \cap \ind_{\K\Z}^{\B} (\omega^r \otimes 1)$ is generated by the $\B$-translates of
\[ \begin{cases}
T([1,x^{r-i}y^i]) & \text{for $0 \leq i \leq r-1$,} \\
T( \sum_{i=0}^{p-1} \lambda_i [ \smat{p & i \\ 0 & 1} , y^r ]) & \text{where $(\lambda_0,\hdots,\lambda_{p-1}) \in V_{r,1}^\perp$}.
\end{cases} \]
\end{lemm}

\begin{proof}
Lemma \ref{formule} above implies that $T([1,x^{r-i}y^i]) \in \ind_{\K\Z}^{\B} (\omega^r \otimes 1)$ if $i \leq r-1$ and hence likewise for the $\B$-translates of those vectors. We therefore only need to determine when a vector of the form $T(\sum_\alpha [b_\alpha, \lambda_\alpha y^r])$ belongs to $\ind_{\K\Z}^{\B} (\omega^r \otimes 1)$. If $v$ is a vector $v = \sum_{\beta,\delta} \sum_{i=0}^{p-1} \lambda_{\beta,\delta,i} [g_{p^{-1} \beta + p^{-1} i ,\delta}, y^r]$ (note that $A = \coprod_{i=0}^{p-1} p^{-1} A + p^{-1} i$), then we have
\[ T(v) = \sum_{\beta,\delta} g_{\beta,\delta+1} \cdot T \left( \lambda_{\beta,\delta,0} [ \smat{1 & 0 \cdot p^{-1} \\ 0 & p^{-1}},y^r] + \cdots + \lambda_{\beta,\delta,p-1} [ \smat{1 & (p-1) \cdot p^{-1} \\ 0 & p^{-1}},y^r] \right), \]
so that by lemma \ref{formule}, the set of vectors $v$ such that $T(v) \in \ind_{\K\Z}^{\B} (\omega^r \otimes 1)$ is generated by the $\B$-translates of the $v_\lambda =  \sum_{i=0}^{p-1} \lambda_i [ \smat{1 & p^{-1}i \\ 0 & p^{-1}} , y^r ]$ such that $T(v_\lambda) \in \ind_{\K\Z}^{\B} (\omega^r \otimes 1)$. Lemma \ref{formule} shows that this is the case if and only if $\sum_{i=0}^{p-1} \lambda_i [ \smat{1 & i \\ 0 & 1} , y^r ] \in \ind_{\K\Z}^{\B} (\omega^r \otimes 1)$ and so if and only if $\sum_{i=0}^{p-1} \lambda_i (ix+y)^r \in E \cdot x^r$ which is equivalent to $(\lambda_0,\hdots,\lambda_{p-1}) \in V_{r,1}^\perp$ since the vector space generated by the sequences $(0^\ell,1^\ell,\hdots,(p-1)^\ell)$ for $0 \leq \ell \leq r-1$ is $V_{r,1}$ (here $0^0=1$). Finally, we multiply the resulting $v_\lambda$ by $\smat{p & 0 \\ 0 & p}$.
\end{proof}

\begin{lemm}\label{enghecke}
If $r=0$, then $T(\ind_{\K\Z}^{\B} (1 \otimes 1))$ is generated by the $\B$-translates of
\[ \pmat{1 & 0\\0 & p}  [1,1] + \sum_{j =0}^{p-1} \pmat{p & j \\ 0 & 1} [1,1]  \] 
and if $r \geq 1$, then $T (\ind_{\K\Z}^{\B} \Sym^r E^2) \cap \ind_{\K\Z}^{\B} (\omega^r \otimes 1)$ is generated by the $\B$-translates of
\[ \sum_{j =0}^{p-1} \lambda_j \pmat{p & j \\ 0 & 1} [1,x^r],  \]
for $(\lambda_0,\hdots,\lambda_{p-1}) \in V_{r,1}$ and of
\[Ê\sum_{i=0}^{p-1} \mu_i i^r [1,x^r] + \sum_{i=0}^{p-1} \mu_i \pmat{p &  i \\ 0 & 1} \sum_{j =0}^{p-1} (-j)^r \pmat{p & j \\ 0 & 1} [1,x^r], \]
where $(\mu_0,\hdots,\mu_{p-1}) \in V_{r,1}^\perp$.
\end{lemm}

\begin{proof}
Since $\ind_{\K\Z}^{\B} (1 \otimes 1)$ is generated by the $\B$-translates of $[1,1]$, the space $T(\ind_{\K\Z}^{\B} (1 \otimes 1))$ is generated by the $\B$-translates of $T([1,1]) = \smat{1 & 0\\0 & p}  [1,1] + \sum_{j =0}^{p-1} \smat{p & j \\ 0 & 1} [1,1]$ which proves the first part.

If $r \geq 1$, then lemma \ref{formule} tells us that $T([1,x^{r-i}y^i]) =  \sum_{j =0}^{p-1}  \smat{p & j \\ 0 & 1} [1, (-j)^i x^r]$ for $i \leq r-1$ and that
\[ T\left( \sum_{i=0}^{p-1} \mu_i [ \smat{p & i \\ 0 & 1} , y^r ]\right) = \sum_{i=0}^{p-1} \mu_i [\smat{1 & i \\ 0 & 1},y^r] + \sum_{i=0}^{p-1} \mu_i \smat{p & i \\ 0 & 1} \sum_{j =0}^{p-1} (-j)^r \smat{p & j \\ 0 & 1} [1,x^r]. \]
The condition $(\mu_0,\hdots,\mu_{p-1}) \in V_{r,1}^\perp$ implies that $\sum_{i=0}^{p-1} \mu_i [\smat{1 & i \\ 0 & 1},y^r] = \sum_{i=0}^{p-1} \mu_i i^r [1,x^r]$ and we are done by lemma \ref{noyr}.
\end{proof}

\begin{theo}\label{ssgisphl}
If $1 \leq h \leq p-1$, then we have an isomorphism of representations of $\B$
\[Ê\Pi_2(h,\sigma) \simeq \frac{\ind_{\KK\Z}^{\G} \Sym^{h-1} E^2}{T(\ind_{\KK\Z}^{\G} \Sym^{h-1} E^2)} \otimes (\chi \circ \det). \]
\end{theo}

\begin{proof}
First of all, we have
\[ (\ind_{\KK\Z}^{\G} \Sym^{h-1} E^2) / T \simeq (\ind_{\K\Z}^{\B} \Sym^{h-1} E^2 )/T \]
by lemma \ref{iwasawa}, so we work with the latter space. We can twist both sides by the inverse of $\chi \circ \det$ so that $\sigma = \omega^{h-1} \otimes 1$ by remark \ref{bze}. Given proposition \ref{isombor}, all we need to check is that if
\[ T(h,\sigma)  = T (\ind_{\K\Z}^{\B} \Sym^{h-1} E^2) \cap \ind_{\K\Z}^{\B} \sigma, \] 
then $T(h,\sigma)$ contains $S_2(h,\sigma)$. The space generated by the vectors $(\lambda_0,\hdots,\lambda_{p-1}) \in V_{h-1,1}$ and by $(0^{h-1},1^{h-1},\hdots,(p-1)^{h-1})$ is $V_{h,1}$ so that by lemma \ref{enghecke}, $T(h,\sigma)$ contains all of the elements
\[Ê\sum_{i=0}^{p-1} \mu_i i^{h-1} [\Id] + \sum_{i=0}^{p-1}  \sum_{j =0}^{p-1} \mu_i \nu_j \left[\pmat{p^2 & pj+i \\ 0 & 1} \right], \]
with $\mu \in V_{p-h+1,1}$ and $\nu \in (-1)^{h-1} (h-1)! v_{h-1,1} + V_{h-1,1}$. If we take $\mu_i = \binom{-i}{p-h}$ and $\nu_j = (h-1)! \binom{-j-1}{h-1}$, then the fact that
\[ \binom{-i}{p-h}\binom{-j-1}{h-1} = \binom{-pj-i}{p(h-1)+p-h} = \binom{-pj-i}{h(p-1)} \]
shows that $T(h,\sigma)$ contains $S_2(h,\sigma)$. 
\end{proof}

\section*{List of notations}

Here is a list of the main notations of the article, in the order in which they appear.

\begin{itemize}
\item[Introduction:] $\gal$; $\B$; $E$; $\K$; $\Z$; primitive $h$; $\inert$; $T_{\pm}$; $T$;
\item[\S \ref{linalg}:] $V_n$; $v_{k,n}$; $V_{k,n}$; $\Delta$; $\mu_a$; 
\item[\S \ref{twtr}:] $g_{\beta,\delta}$; $\sigma$; $\ind_{\K\Z}^{\B} \sigma$; $[g]$; $\alpha(\beta,\delta)$; support; level; $n$-block; initial $n$-block; $\I$; $\tau_k$;
\item[\S \ref{sirb}:] $w_{\ell,n}$; $\lambda$; $S_n(h,\sigma)$; $\Pi_n(h,\sigma)$; $i_k$; $h_k$; $\B^+$;
\item[\S \ref{fontaine}:] $\et^+$; $\et$; $\eps$; $X$; $\hal$; $\Gamma$; $\dfont(W)$; $\omega_n$; $\omega$; $\mu_\lambda$; $\ind(\omega_n^h)$; 
\item[\S \ref{colmez}:] $\psi$; $\Omega(W)$; $\theta$; 
\item[\S \ref{dimde}:] $T_{\pm}$; $T$;
\item[\S \ref{ssgresb}:] $\Sym^r(E^2)$.
\end{itemize}

\vspace{20pt}

\noindent\textbf{Acknowledgements:} It's a pleasure to thank Christophe Breuil for suggestions which resulted in this work, as well as for answering many of my subsequent questions. The ideas of Fontaine, Colmez and Breuil are the foundation for this work, and I am fortunate to have benefited from their guidance and insights throughout the years. Finally, I thank Michel Gros, Florian Herzig, Sandra Rozensztajn and Mathieu Vienney for their helpful comments on earlier versions, and the referee for his meticulous work which significantly improved the quality of this article.

\bibliographystyle{smfalpha}
\bibliography{borelmodp}
\end{document}